\documentclass[11pt]{article}
\usepackage{subcaption}
\usepackage{booktabs}
\usepackage{definitions}

\begin{document}

\title{Regularized Online Allocation Problems: Fairness and Beyond}


\author{Santiago R. Balseiro\thanks{Columbia Business School and Google Research.}
\and
Haihao Lu\thanks{University of Chicago Booth School of Business. Part of the work was done at Google.}
\and
Vahab Mirrokni \thanks{Google Research.}}

\date{October 27, 2021}
\maketitle

\begin{abstract}
Online allocation problems with resource constraints have a rich history in operations research. In this paper, we introduce the \emph{regularized online allocation problem}, a variant that includes a non-linear regularizer acting on the total resource consumption. In this problem, requests repeatedly arrive over time and, for each request, a decision maker needs to take an action that generates a reward and consumes resources. The objective is to simultaneously maximize additively separable rewards and the value of a non-separable regularizer subject to the resource constraints. Our primary motivation is allowing decision makers to trade off separable objectives such as the economic efficiency of an allocation with ancillary, non-separable objectives such as the fairness or equity of an allocation. We design an algorithm that is simple, fast, and attains good performance with both stochastic i.i.d.~and adversarial inputs. In particular, our algorithm is asymptotically optimal under stochastic i.i.d. input models and attains a fixed competitive ratio that depends on the regularizer when the input is adversarial. Furthermore, the algorithm and analysis do not require convexity or concavity of the reward function and the consumption function, which allows more model flexibility. Numerical experiments confirm the effectiveness of the proposed algorithm and of regularization in an internet advertising application.

\end{abstract}

\vspace{0.2cm}
\setstretch{1.5}
\section{Introduction}

Online allocation problems with resource constraints have abundant real-world applications and, as such, have been extensively studied in computer science and operations research. Prominent applications can be found in internet advertising~\citep{Feldman2010}, cloud computing~\citep{badanidiyuru2018bwk}, {allocation of food donations to food banks~\citep{lien2014sequential}, rationing of a social goods during a pandemic~\citep{mansahdi2021fair}, ride-sharing~\citep{nanda2020balancing}, workforce hiring~\citep{dickerson2019balancing}, inventory pooling~\citep{jiang2019achieving} among others.} 

The literature on online allocation problems focuses mostly on optimizing additively separable objectives such as the total click-throughout rate, revenue, or efficiency of the allocation. In many settings, however, decision makers are also concerned about ancillary objectives such as fairness across agents, group-level fairness, avoiding under-delivery of resources, balancing the load across servers, or avoiding saturating resources. These metrics are, unfortunately, non-separable and cannot be readily accommodated by existing algorithms that are tailored for additively separable objectives. Thus motivated, in this paper, we introduce the \emph{regularized online allocation problem}, a variant that includes a non-linear regularized acting on the total resource consumption. The introduction of a regularizer allows the decision maker to simultaneously maximize an additively separable objective together with other metrics such as  fairness and load balancing that are non-linear in nature.

More formally, we consider a finite horizon model in which requests arrive repeatedly over time. The decision maker is endowed with a fixed amount of resources that cannot be replenished. Each arriving request is presented with a non-linear reward function and a consumption function. After observing the request, the decision makers need to take an action that generates a reward and consumes resources. The objective of the decision maker is to maximize the sum of the cumulative reward and a regularizer that acts on the total resource consumption. (Our model can easily accommodate a regularizer that acts on other metrics such as, say, the cumulative rewards by adding dummy resources.)

Motivated by practical applications, we consider an incomplete information model in which requests are generated from an input model that is \emph{unknown to the decision maker}. That is, when a request arrives, the decision maker observes the reward function and consumption function of the request before taking an action, but does not get to observe the reward functions and consumption matrices of future requests until their arrival. The objective of this paper is to design simple algorithms that attain good performance relative to the best allocation when all requests are known in advance. 

\subsection{Main Contributions}

The main contributions of the paper can be summarized as follows:
\begin{enumerate}
    \item We introduce the \emph{regularized online allocation problem} and discuss various practically relevant economic examples (Section~\ref{sec:formulation}). Our proposed model allows the decision maker to simultaneously maximize an additively separable objective together with an non-separable metric, which we refer to as the \emph{regularizer}. Furthermore, we do not assume the reward functions nor the consumption functions to be convex/concave---a departure from related work.
    
    \item We derive a dual problem for the regularized online allocation problem, and study its geometric structure as well as its economic implications (Section~\ref{sec:dual-problem} and Section~\ref{sec:algorithm-examples}). Due to the existence of a regularizer, the dual variables (i.e., the opportunity cost of resources) no longer necessarily lie in the positive orthant and can potentially be negative, unlike traditional online allocation problems.
    \item We propose a dual mirror descent algorithm for solving regularized online allocation problems that is simple and efficient, can better capture the geometry of the dual feasible region, and achieves good performances for various input models:
\begin{itemize}
    \item Our algorithm is efficient and simple to implement. In many cases, the update rule can be implemented in linear time and there is no need to solve auxiliary  optimization problems on historical data as in other methods in the literature.
    \item The dual feasible region can become complicated for regularized problems. By suitably picking the reference function in our mirror descent algorithm, we can adjust the update rule to better capture the geometry of the dual feasible set and obtain more tractable projection problems.
    \item When the input is stochastic and requests are drawn are independent and identically drawn from a distribution that is unknown to the decision maker, our algorithm attains a regret of the order $O(T^{1/2})$ relative to the optimal allocation in hindsight where $T$ is the number of requests (Section~\ref{sec:iid-model}). This rate is unimprovable under our minimal assumptions on the input. On the other extreme, when requests are adversarially chosen, no algorithm can attain vanishing regret, but our algorithm is shown to attain a fixed competitive ratio that depends on the structure of the regularizer, i.e., it guarantees a fixed fraction of the optimal allocation in hindsight (Section~\ref{sec:adversarial}). 
    
    \item We numerically evaluate our algorithm on an internet advertising application using a max-min fairness regularizer. Our experiments confirm that our proposed algorithm attains $O(T^{1/2})$ regret under stochastic inputs  as suggested by our theory. { By varying the weight of the regularizer, we can trace the Pareto efficient frontier between click-through rates and fairness. An important managerial takeaway is that fairness can be significantly improved while only reducing click-through rates by a small amount.}
\end{itemize}
\end{enumerate}

Our model is general and applies to online allocation problems across many sectors. We briefly discuss some salient applications. In display advertising, a publisher typically signs contracts with many advertisers agreeing to deliver a fixed number of impressions within a limited time horizon. Impressions arrive sequentially over time and the publisher needs to assign, in real time, each impression to one advertiser so as to maximize metrics such as the cumulative click-through rate or the number of conversions while satisfying contractual agreements on the number of impressions to be delivered~\citep{Feldman2010}. Focusing on maximizing clicks can lead to imbalanced outcomes in which advertisers with low click-through rates are not displayed or to showing ads to users who are most likely to click without any equity considerations~\citep{cainmiller2015}. { We propose mitigating these issues using a fairness regularizer. Additionally, by incorporating a regularizer, publishers can punish under-delivery of advertisements---a common desiderata in internet advertising markets. Another related advertising application is bidding in repeated auctions with budgets constraints, where the regularizer can penalize the gap between the true spending and the desired spending~\citep{grigas2021optimal} or model advertisers' preferences to reach a certain distribution of impressions over a target population of users~\citep{celli2021parity}.}

In cloud computing, jobs arriving online need to be scheduled to one of many servers. Each job consumes resources from the server, which need to be shared with other jobs. The scheduler needs to assign jobs to servers to maximize metrics such as the cumulative revenue or efficiency of the allocation~\citep{xu2013dynamic}. When jobs' processing times are long compared to their arrival rates, this scheduling problem can be cast as an online allocation problem~\citep{badanidiyuru2018bwk}. { Regularization allows to balance the load across resources to avoid saturating them, to incorporate returns to scale in resource utilization,} or to guarantee fairness across users~\citep{bateni2018fair}.
    

{
Finally, in humanitarian logistics, limited supplies need to be distributed to agents repeatedly over time. Maximizing the economic efficiency of the allocation can lead to unfair outcomes in which some agents or groups of agents receive lower relative amount of supplies. For example, during the Covid-19 pandemic, supply-demand imbalance leading to the rationing of medical supplies and allocating resources fairly was a key consideration for decision makers~\citep{emanuel2020fair,mansahdi2021fair}. In these situations, incorporating a fairness regularizer that maximizes the minimum utility across agents (or groups of agents) can lead to more egalitarian outcomes while slightly sacrificing economic efficiency~\citep{bateni2018fair}.
}

\subsection{Related Work}

\begin{table*}[]
\footnotesize
    \centering
    \def\arraystretch{1.2}
    \begin{tabular}{p{6cm}x{2.25cm}x{1.8cm}x{1.75cm}c}
        \toprule
         &  & \bf Resource & \bf Input & \\
         \bf Papers & \bf Objective & \bf constraints & \bf model & \bf Results \\ \midrule
         \citet{Mehta2007JACM, buchbinder2007primaldual, Feldman2009, ball2009toward, ma2021algorithms}&Separable&  Hard & Adversarial & Fixed comp. ratio\\\hline
         \citet{DevanurHayes2009, Feldman2010, Agrawal2014OR, Devanur2019near, Jasin2015unknown, LiYe2019online, li2020simple}&Separable & Hard & Stochastic & Sublinear  regret \\\hline
         \multirow[t]{2}{*}{\citet{mirrokni2012simultaneous, balseiro2020best}}&Separable & Hard & Stochastic/ & Sublinear  regret/\\
         &   & & Adversarial & Fixed comp. ratio \\ \hline         
         \citet{jenatton2016online,AgrawalDevanue2015fast,cheung2020online}&Non-separable & Soft & Stochastic & Sublinear  regret \\\hline
         \citet{devanur2012concave,eghbali2016designing}&Non-separable & Soft & Adversarial & Fixed comp. ratio \\\hline
         \citet{tan2020mechanism} & Non-separable & Hard & Adversarial & Fixed comp. ratio\\\hline
         This paper & Non-separable & Hard & Stochastic/ & Sublinear  regret/ \\
         & & & Adversarial & Fixed comp. ratio \\         \bottomrule
    \end{tabular}
    \caption{Comparison of our work with the existing literature on online allocation problems.}
    \label{tab:literature}
\end{table*}

Online allocation problems have been extensively studied in computer science and operations research literature. Table~\ref{tab:literature} summarizes the differences between our work and the existing literature on online allocation problems.



\textbf{Stochastic Input.} Most work on the traditional online allocation problem with stochastic input models focus on time-separable and linear reward functions, i.e., the case when the reward function is a summation over all time periods and is linear in the decision variable~\citep{DevanurHayes2009, Feldman2010, Agrawal2014OR, Jasin2015unknown, Devanur2019near, LiYe2019online}. The algorithms in these works usually require resolving large linear problems periodically using all samples collected so far to update the dual variables or to estimate the distribution of requests. More recently, \citet{balseiro2020best} studied a simple dual mirror descent algorithm for online allocation problems with separable, non-linear reward functions, which attains $O(T^{1/2})$ regret and has no need of solving any large program using all previous samples. Simultaneously, ~\citet{li2020simple} presented a similar subgradient algorithm that attains $O(T^{1/2})$ regret for linear reward. 
Our algorithm is similar to theirs in that it does not require to solve large problems using all data collected so far. 

Compared to the above works, we introduce a regularizer on top of the traditional online allocation problem. The regularizer can be viewed as a reward that is not separable over time. The existence of such non-separable regularizer makes the theoretical analysis harder compared to traditional dual-based algorithms for problems without regularizers because dual variables, which correspond to the opportunity cost that a unit of resource is consumed, can be negative due to the existence of the regularizer. 
We present a in-depth characterization on the geometry of the dual problem (see Section \ref{sec:characterization_dual}), which allows us to get around this issue in the analysis.

Another related work on stochastic input models is~\citet{AgrawalDevanue2015fast}, where the focus is to solve general online stochastic convex program that allows general concave objectives and convex constraints. When the value of the benchmark is known, they present fast algorithms; otherwise, their algorithms require periodically solving large convex optimization problems with the data collected so far to obtain a good estimate of the benchmark. In principle, our regularized online allocation problem (see Section \ref{sec:formulation} for details) can be reformulated as an instance of the online stochastic convex program presented in~\citet{AgrawalDevanue2015fast}. Such reformulation makes the algorithms proposed in~\citet{AgrawalDevanue2015fast} more complex than ours as they require keeping track of additional dual variables and solving convex optimization programs on historical data (unless the optimal value of the objective is known). Moreover, their algorithms treat resource constraints as soft, i.e., they allow the constraints to be violated and then prove that constrains are violated in expectation, by an amount sublinear in the number of time periods $T$. Instead, in our setting, resource constraints are hard and cannot be violated, which is a fundamental requirement in many applications. Additionally, our proposed algorithm is simple, fast, and does not require estimates of the value of the benchmark nor solving large convex optimization problems. Along similar lines \citet{jenatton2016online} studies a problem with soft resource constraints in which violations of the constraints are penalized using a non-separable regularizer. Thus, different to our paper, they allow the resource constraints to be violated. They provide a similar algorithm to ours, but their regret bounds depend on the optimal solution in hindsight, which, in general, it is hard to control. Finally, \citet{cheung2020online} design similar dual-based online algorithms for non-separable objectives with soft resource constraints when offline samples from the distribution of requests are available.

{
\textbf{Adversarial Inputs.}
Most of the previous works on online allocation problems with adversarial inputs focus on separable objectives. There is a stream of work studying the \emph{AdWords problem}, the case when the reward and the consumption  are both linear in decisions and reward is proportional to consumption~\citep{Mehta2007JACM, buchbinder2007primaldual, Feldman2009,mirrokni2012simultaneous}, revenue management problems~\citep{ball2009toward, ma2021algorithms}, or more general separable objectives with non-linear reward and consumption functions~\citep{balseiro2020best}. 
Compared to the above work, we include a non-separable regularizer. Similar to the stochastic setting, the existence of such non-separable regularizer makes the competitive ratio analysis more challenging than the one without the regularizer. As opposed to the stochastic setting in which vanishing regret can be achieved for most concave regularizers, in the adversarial case the competitive ratio heavily depends on the structure of the regularizer. In particular, \cite{mirrokni2012simultaneous} shows that no online algorithm can obtain both vanishing regret under stochastic input and can attain a data-independent constant competitive ratio for adversarial input for the AdWords problem. While AdWords problem is a special case of our setting, our results do not contradict their findings because our competitive ratio depends on the data, namely, the consumption-budget ratio. 



There have been previous works studying online allocation problems with non-separable rewards in the adversarial setting, but most impose more restrictive assumptions on the objective function, which preclude some of the interesting examples we present in Section~\ref{sec:ex}. For example, \citet{devanur2012concave} study a variant of the AdWords problem in which the objective is concave and monotone over requests but separable over resources. \citet{eghbali2016designing} extend their work by considering more general non-separable concave objectives, but still require the objective to be monotone and impose restrictive assumptions on the reward function. Even though they do not consider hard resource constraints, they provide similar primal-dual algorithms that are shown to attain fixed competitive ratios.} \citet{tan2020mechanism} studies a similar problem as our regularized online allocation setting, but restrict attention to one resource.

\textbf{Fairness Objectives.} While most of the literature on online allocation focuses on maximizing an additively separable objective, other features of the allocation, such as fairness and load balancing, sometimes are crucial to the decision maker. Fairness, in particular, is a central concept in welfare economics. Different reasonable metrics of equity have been proposed and studied: max-min fairness, which maximizes the reward of the worst-off agents~\citep{nash1950bargaining,bansal2006santa}; proportional fairness, which makes sure that there is no alternative allocation that can lead to a positive aggregate proportional change for each agent~\citep{azar2010allocate,bateni2018fair}; or $\alpha$-fairness, which generalizes the previous notions~\citep{mo2000fair,bertsimas2011price,bertsimas2012efficiency}, and allows to recover max-min fairness and proportional fairness as special cases when $\alpha=\infty$ and $\alpha=1$, respectively. The above line of work focuses on optimizing the fairness of an allocation problem (without considering the revenue) in either static settings, adversarial settings, or stochastic settings that are different to ours. In contrast, our framework is concerned with maximizing an additively separable objective but with an additional regularizer corresponding to fairness (or other desired ancillary objective).

 A related line of work studies online allocation problems with fairness, equity, or diversity objectives in Bayesian settings in which the distribution of requests is known to decision makers. \cite{lien2014sequential} and \cite{mansahdi2021fair} study the sequential allocation of scarce resources to multiple agents (requests) in nonprofit applications such as distributing food donations across food banks or medical resources during a pandemic. They measure the equity of an allocation using the max-min fill rate across agents, where the fill rate is given by the ratio of allocated amount to observed demand, and design algorithms when the distribution of demand is non-stationary and potentially correlated across agents. In contrast to our work, fairness is measured across requests as opposed to across resources. Relatedly, \citet{ahmed2017diverse} and \citet{dickerson2019balancing} design approximation algorithms for stochastic online matching problems with submodular objective functions. Submodular objectives allow the authors to incorporate the fairness, relevance, and diversity of an allocation in applications related to advertising, recommendation systems, and workforce hiring. 
 \cite{ma2020group} consider stochastic matching problems with a group-level fairness objective, which involves maximizing the minimum fraction of agents served across various demographic groups. The authors design approximation algorithms with fixed competitive ratios when the distribution of requests is known in advance. Group-level fairness objectives can be incorporated in our model by adding a dummy resource for each group and then using a max-min fairness regularizer. \citet{nanda2020balancing} provide an algorithm that can optimally tradeoff the profit and group-level fairness of an allocation for online matching problems and present an application to ride-sharing. Our algorithm is asymptotically optimal for group-level fairness regularizers when the number of groups is sublinear in the number of requests.  \cite{agrawal2018proportional} presents a simple proportional allocation algorithm with high entropy to increase the diversity of the allocation for the corresponding offline problem where all requests are known to the decision maker beforehand.

\textbf{Preliminary Version.} In a preliminary conference proceedings version of this paper~\citep{balseiro2021regularized}, we present the basic setup of the regularized online allocation problem and the regret analysis for online subgradient descent under stochastic i.i.d.~inputs. This paper extends the preliminary proceeding version by (i) introducing a competitive analysis for adversarial inputs, (ii) providing an in-depth study of the geometric structure of the dual problem, (iii) modifying the algorithm and its ensuing analysis to include online mirror descent, which leads to iterate updates with closed-form or simpler solutions.

\vspace{0.2cm}
\section{Problem Formulation and Examples}\label{sec:formulation}
We consider a generic online allocation problem with a finite horizon of $T$ time periods, resource constraints, and a concave regularizer $r$ on the resource consumption. At time $t$, the decision maker receives a request $\gamma_t = (f_t,b_t, \cX_t) \in \mathcal S$ where $f_t\in \RR^d\rightarrow\RR_+$ 
is a non-negative (and potentially non-concave) reward function and $b_t:\cX_t\rightarrow\RR^m_+$ is an non-negative (and potentially non-linear) resource consumption function, and $\cX_t \subset \RR^d_+$ is a (potentially non-convex or integral) decision set. We denote by $\mathcal S$ the set of all possible requests that can be received. After observing the request, the decision maker takes an action $x_t\in \cX_t\subseteq \RR^d$ that leads to reward $f_t(x_t)$ and consumes $b_t(x_t)$ resources. The total amount of resources is $T \rho$, where $\rho\in \RR^m_{+}$ is a positive resource constraint vector. The assumption $b_t(x)\ge 0$ implies that we cannot replenish resources once they are consumed. We assume that $0\in \cX_t$ and $b_t(0) = 0$. The above assumption implies that we can always take a void action by choosing $x_t=0$ to make sure we do not exceed the resource constraints. This guarantees the existence of a feasible solution.

\begin{ass}[Assumptions on the requests]\label{ass:p} There exists $\ubf\in\RR_{+}$ and $\ubb \in \RR_+$ such that for all requests $(f,b,\cX) \in \cS$ in the support, it holds $0\le f(x)\le \ubf$ for all $x\in\cX$, $b(x)\ge 0$ and $\|b(x)\|_{\infty}\le \ubb$ for all $x\in\cX$.
\end{ass}

The upper bounds $\ubf$ and $\ubb$ impose regularity on the space of requests and will appear in the performance bounds we derive. They need not be known by the algorithm. We denote by $\|\cdot\|_{\infty}$ the $\ell_\infty$ norm and $\|\cdot\|_{1}$ the $\ell_1$ norm of a vector. 



An online algorithm $A$ makes, at time $t$, a real-time decision $x_t$ based on the current request $\gamma_t = (f_t, b_t, \cX_t)$ and the previous history $\cH_{t-1}:=\{\gamma_s, x_s\}_{s=1}^{t-1}$, i.e.,
\begin{equation}\label{eq:al_update}
x_t =A(\gamma_t|\cH_{t-1})  \ .
\end{equation}
Note that when taking an action, the algorithm observes the current request but has no information about future ones. For example, in display advertising, publishers can estimate, based on the attributes of the visiting user, the click-through rates of each advertiser before assigning an impression. However, the click-through rates of future impressions are not known in advance as these depend on the attributes of the unknown, future visitors. Similar information structures arise in cloud computing, ride-sharing, etc.

For notational convenience, we utilize $\vgamma = (\gamma_1,\ldots,\gamma_T)$ to denote the inputs  over time $1,\ldots,T$.  We define the \emph{regularized reward} of an algorithm for input $\vgamma$ as 
\[
    R(A | \vgamma) = \sum_{t=1}^T f_t(x_t) + T r\left(\frac 1 T \sum_{t=1}^T b_t(x_t) \right) \,,
\]
where $x_t$ is computed by \eqref{eq:al_update}. Moreover, algorithm $A$ must satisfy constraints $\sum_{t=1}^{T} b_{t} (x_{t}) \le \rho T$ and $x_{t}\in \cX_t$ for every $t\le T$. { While the regularizer $r$ is imposed on the average consumption, our model is flexible and can accommodate regularizers that act on other quantities such as total rewards by introducing dummy resource constraints (see Example~\ref{ex:SC} for more details).} 

The baseline we compare with is the reward of the optimal solution when the request sequence $\vgamma$ is known in advance, which amounts to solving for an allocation that maximizes the regularized reward subject to the resource constraints under full information of all requests:
\begin{align}\label{eq:OPT}
\OPT(\vec \gamma)=
	\max_{x: x_t\in \cX_t} & \sum_{t=1}^T f_t(x_t) + T \roneT     \\
	\text{s.t.}     & \sum_{t=1}^T b_t(x_t) \le T\rho\,. \nonumber
\end{align}
The latter problem is referred to as the offline optimum in the computer science or hindsight optimum in the operations research literature.

Our goal is to design an algorithm $A$ that attains good performance, while satisfying the resource constraints, for the stochastic i.i.d.~input model and the adversarial input model. The notions of performance depend on the input models and, as such, they are formally introduced in Sections~\ref{sec:iid-model} and~\ref{sec:adversarial} where the input models are presented.



\subsection{Examples of the Regularizer}\label{sec:ex}

We now present some examples of the regularizer. To guarantee that the total rewards are non-negative we shift all regularizers to be non-negative. Many of our results can also be extended to negative regularizers. First, by setting the regularizer to zero, we recover a standard online allocation problem.
    \begin{example}[No Regularizer]\label{ex:no_regularizer}
	When the regularizer is $r(a) = 0$, we recover the non-regularized online allocation problem.
\end{example}

Our next example allows for max-min fairness guarantees, which have been studied extensively in the literature~\citep{nash1950bargaining,bansal2006santa,mo2000fair,bertsimas2011price,bertsimas2012efficiency}. Here we state the regularizer in terms of consumption, which seeks to maximize the minimum relative consumption across resources. In many settings, however, it is reasonable to state the fairness regularizer in terms of other quantities such as the cumulative utility of individual agents or the amount of resources delivered to different demographic groups, i.e., group-level fairness. As discussed in Example~\ref{ex:SC}, such regularizers can be easily accommodated by introducing dummy resource constraints.

\begin{example}[Max-min Fairness]\label{ex:minimal_cons2}
	The regularizer is defined as $r(a) = \lambda \min_{j} (a_j/\rho_j)$, i.e., the minimum relative consumption. This regularizer imposes fairness on the consumption between different resources, making sure that no resource gets allocated a too-small number of units. Here $\lambda > 0$ is a parameter that captures importance of the regularizer relative to the rewards.
\end{example}



In applications like cloud computing, the load should be balanced across machines to avoid congestion. The following regularizer is reminiscent of the makespan objective in machine scheduling. The load balancing regularizer seeks to maximize the minimum relative availability across resources or, alternatively, minimize the maximum relative consumption across resources.

\begin{example}[Load Balancing]\label{ex:minimal_cons}
	The regularizer is defined as $r(a) =  \lambda \min_{j} \big( (\rho_j - a_j)/\rho_j \big)$, i.e., the minimum relative resource availability. This regularizer guarantees that consumption is evenly distributed across resources by making sure that no resource is too demanded.
\end{example}

In some settings, like cloud computing, firms would like to avoid saturating resources because the costs of utilizing the resource exhibit a decreasing return to scale or to keep some idle capacity to preserve flexibility. The next regularizer allows to capture situations in which it is desired to keep resource consumption below a target.

\begin{example}[Below-Target Consumption]\label{ex:hinge_loss}
	The regularizer is defined as $r(a)=\sum_{j=1}^m c_j \min(\rho_j-a_j, \rho_j-t_j)$ with thresholds $t_j \in [0,\rho_j]$ and unit reward $c_j$. This regularizer can be used when the decision maker receives a reward $c_j$ for reducing each unit of resource below a target $t_j$. 
\end{example}

To maximize reach, internet advertisers, in some cases, prefer that their budgets are spent as much as possible or their reservation contracts are delivered as many impressions as possible. We can incorporate these features by rewarding the decision making whenever the target is met.

\begin{example}[Above-Target Consumption]\label{ex:hinge_loss2}
	The regularizer is defined as $r(a)=\sum_{j=1}^m c_j \min(a_j, t_j)$ with thresholds $t_j \in [0,\rho_j]$ and unit reward $c_j$. This regularizer can be used when the decision maker would like to consume at least $t_j$ units of resource $j$ and the decision maker is offered a unit reward of $c_j$ for each unit closer to the target.
\end{example}

The regularizers in the previous examples act exclusively on resource consumption. The next example shows that by adding dummy resource constraints that never bind, it is possible to incorporate regularizers that act on other quantities.


\begin{example}[Santa Claus Regularizer from \citealt{bansal2006santa}]\label{ex:SC}
	There are $k$ agents (i.e., the kids) and the decision maker (i.e., Santa Claus) intends to make sure the minimal reward (i.e., the presents) of each agent is not too small. Here we consider	an additional reward function $q_t: \cX_t \rightarrow \RR_+^k$ that gives the reward vector for each agent under decision $x_t$. To reduce such problem to our setting, we first add auxiliary resource constraints $\sum_{t=1}^T q_t( x_t)\le T \bar q e$ that never bind, where $\bar q$ is a uniform upper bound on the additional rewards and $e\in\RR^k$ is the all-one vector. Then, the regularizer is $r\pran{\frac{1}{T}\sum_{t=1}^T q_t(x_t), \frac{1}{T}\sum_{t=1}^T b_t(x_t) }=\frac 1 {T} \min_j \sum_{t=1}^T (q_t(x_t))_j$.
\end{example}

\section{The Dual Problem}\label{sec:dual-problem}

Our algorithm is of dual-descent nature and, thus, the Lagrangian dual problem of \eqref{eq:OPT} and its constraint set play a key role. We construct a Lagrangian dual of \eqref{eq:OPT} by introducing an auxiliary variable $a \in \RR^m$ satisfying $a = \frac 1 T \sum_{t=1}^T b_t(x_t)$ that captures the \emph{time-averaged resource consumption} and then move the constraints to the objective using a vector of Lagrange multipliers $\mu \in \RR^m$. The auxiliary primal variable $a$ and the dual variable $\mu$ are key ingredients of our algorithm. For $\mu \in \RR_+^m$ we define
\begin{equation}\label{eq:conjugate}
f_t^*(\mu):=\sup_{x\in \cX_t} \{f_t(x)-\mu^{\top} b_t(x)\}\,,
\end{equation}
as the optimal opportunity-cost-adjusted reward of request $\gamma_t$. The function $f_t^*(\mu)$ is a generalization of the convex conjugate of the function $f_t(x)$ that takes account of the consumption function $b_t(x)$ and the constraint space $\cX_t$. We define
\begin{equation}
    r^*(-\mu)=\sup_{a\le \rho}\{r(a)+\mu^\top a\}\,,
\end{equation}
as the conjugate function of the regularizer $r$.
 Define $\cD=\{\mu \in \mathbb R^m \mid  r^*(-\mu) <+\infty\}$ as the set of dual variables for which the conjugate of the regularized is bounded. For a fixed input $\vgamma$, define the Lagrangian dual function $D(\mu | \vgamma):\cD\rightarrow\RR$ as
$$D(\mu | \vgamma ):=\sum_{t=1}^T f_t^*(\mu)  +T r^*(-
\mu)\ .$$ Then, we have by weak duality that $D(\mu|\vgamma)$ provides an upper bound on $\OPT(\vgamma)$.  To see this, note that for every $\mu \in \cD$ we have that
\begin{align}\label{eq:obtain_dual}
\begin{split}
 \OPT(\vgamma)
= &  \mpran{
	\begin{array}{cl}
	\max_{x_t\in \cX_t, a\le \rho} & \sum_{t=1}^T f_t(x_t) +T r(a)      \\
	\text{s.t.}     & \sum_{t=1}^T b_t(x_t) = T a
	\end{array}}\\
\le &  \mpran{\max_{x_t\in \cX_t, a\le \rho} \npran{\sum_{t=1}^T f_t(x_t) + T r(a)+ T\mu^{\top} a - \mu^{\top} \sum_{t=1}^T  b_t(x_t) }} \\
= &  \sum_{t=1}^T f_t^*(\mu)  +T r^*(-
\mu) 
= D(\mu | \vgamma ) \ ,
\end{split}
\end{align}
where  the first inequality is because we relax the constraint $\sum_{t=1}^T b_t(x_t) = T a$, and the last equality utilizes the definition of $r^*$ and $f^*$.


\subsection{Characterization of the Dual Problem}\label{sec:characterization_dual}


In this section, we prove some preliminary properties that elucidate how the choice of a regularizer impacts the geometry of the dual feasible region $\cD$, which is important for our algorithm and its ensuing analysis. Throughout the paper we assume the following conditions on the regularizers. As we shall see in Section~\ref{sec:algorithm-examples}, these conditions are satisfied by all examples discussed in the paper.

\begin{ass}[Assumptions on the regularizer $r$]\label{ass:r2} We assume the regularizer $r$ and the set $\cD=\{\mu \in \mathbb R^m \mid  r^*(-\mu) <+\infty\}$ satisfy:
	\begin{enumerate}

		\item{\color{black} The function $r(a)$ is $L$-Lipschitz continuous in the $\|\cdot\|_{\infty}$-norm on its effective domain, i.e., $|r(a_1) - r(a_2)| \le L \| a_1 - a_2\|_{\infty}$ for any $a_1,a_2 \le \rho$.}
		
		\item{\color{black} There exists a constant $\bar a$ such that for any $\mu\in\cD$, there exists $a\in\arg\max_{a\le \rho}\{ r(a)+\mu^\top a\}$ such that $\|a\|_{\infty}\le \bar a$.}
		
		\item {There exists $\ubr$ such that $0\le r(a)\le \ubr$ for all $0 \le a \le \rho$.}
		
		\item The function $r(a)$ is concave.
		
	\end{enumerate}
\end{ass}

The first part imposes a continuity condition on the regularizer, which is a standard assumption in online optimization. We choose to use the $\ell_\infty$ norm in order to be consistent with Assumption~\ref{ass:p}, where we chose this norm to measure consumption (i.e., $\|b(x)\|_{\infty}\le \ubb$). The second part guarantees that the conjugate of the regularizer has a bounded subgradient, which is a technical assumption required for our analysis, while the last part imposes that the regularizer is non-negative and upper-bounded. The assumption that the regularizer is non-negative is required for the analysis of our algorithm when the input is adversarial (as competitive ratios are not well defined when rewards can be negative). Our results for stochastic inputs hold without this assumption. 


\begin{figure}
    \centering
    \begin{tikzpicture}[scale=2]
        \draw[->] (-1,0) -- (2,0) node[right] {$\mu_1$};
        \draw[->] (0,-1) -- (0,2) node[above] {$\mu_2$};
        
        \fill[domain=0.125:2,smooth,variable=\x,blue!10] plot ({\x},{0.25/\x})  -- (2,2) -- cycle;
        
        \draw[thick,domain=0.125:2,smooth,variable=\x,blue] plot ({\x},{0.25/\x}) node[above] {$\cD$};

        \draw[-,dashed] (0.75,2) -- (0.75,0.75) -- (2,0.75) node[above] {\small $\mu \ge L e$};
        \draw[-,dashed] (-0.75,2) -- (-0.75,0) -- (0, -0.75) -- (2,-0.75) node[above] {\small$\sum_{j=1}^m(\mu_j)^-\le L$};      
        
        \filldraw (0.75,0.75) circle (1pt) node[below] {\small$L e$};
        \filldraw (-0.75,-0.75) circle (1pt) node[below] {\small$-L e$};
    \end{tikzpicture}
    \caption{Feasible region $\cD$ for dual problem (shaded region).}
    \label{fig:dual_feasible_region}
\end{figure}
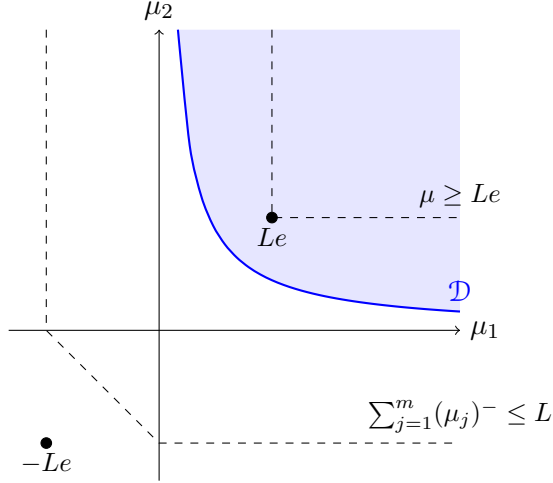

The next result provides a partial geometric characterization of the dual feasible region. Figure \ref{fig:dual_feasible_region} plots the domain $\cD$ for a typical two-dimensional regularizer.

\begin{lem}\label{lem:mu_negative}
The following hold:
\begin{enumerate}[(i)]
\item The set $\cD$ is convex and positive orthant is inside the recession cone of $\cD$, i.e., $\RR^m_+ \subseteq  \emph{recc}(\cD)$.
\item $L e\in \cD$, where $e$ is the all-one vector in $\RR^d$.
\item Let $a \in \arg\max_{a\le \rho} \{r(a)+\mu^\top a\}$. If $\mu_j \ge L$, then $a_j = \rho_j$.
\item For any $\mu\in\cD$, we have that $\sum_{j=1}^m(\mu_j)^-\le L$, where $(\mu_j)^- = \max\{-\mu_j,0\}$ is the negative part of $\mu_j$.

\end{enumerate}
\end{lem}

Property (i) shows that the feasible set is convex, which is useful from an optimization perspective, and unbounded as, for every point, all coordinate-wisely larger points lie in it. The Lipschitz continuity constant of the regularizer plays a key role on the geometry of the dual feasible set. In the case of no regularizer, we have $L=0$ and the properties (ii) and (iv) give that the dual feasible set is the non-negative orthant $\cD = \RR_+^m$. This recovers the well-known result that the opportunity cost of consuming a unit of resource is non-negative in allocation problems. As we shall see, with the introduction of a regularizer the opportunity costs might be  negative, which could induce a decision maker to consume resources even when these generate no reward. For example, in the case of max-min fairness, a resource with low relative consumption might be assigned a negative dual variable to encourage the decision maker to increase its consumption and, thus, increase fairness. Property (iv) limits how negative dual variables can get. In particular, it implies that $\mu_j \ge -L$ for every resource $j$. In other cases, the regularizer can force the dual variables to be larger than zero to reduce the rate of resource consumption. For example, in the case of load balancing, setting the opportunity costs to zero can lead to resource saturation and, as a result, the dual feasible set $\cD$ is a strict subset of the non-negative orthant. Property (ii) upper bounds the minimum value that dual variables can take. Taken together, properties (ii) and (iii) provide bounds on the lower boundary of the dual feasible region in terms of the Lipschitz constant. See Figure~\ref{fig:dual_feasible_region} for an illustration. We defer an explanation of property (iv) until the next section.

\section{Algorithm}\label{sec:algorithm}
\begin{algorithm}[tb]
	
	\SetAlgoLined
	{\bf Input:} Initial dual solution $\mu_0 \in \cD$, total number of time periods $T$, initial resources $B_0=T\rho$, reference function $h(\cdot): \RR^m\rightarrow \RR$, and step-size $\eta$. \\
	\For{$t=0,\ldots,T-1$}{
		Receive $(f_t, b_t, \cX_t)$.\\
		Make the primal decision and update the remaining resources:
		\begin{align}
		\tilde{x}_{t} &= \arg\max_{x\in\cX_t}\{f_{t}(x)-\mu_{t}^{\top} b_{t} (x)\} \ ,\label{eq:primal_decision} \\
		x_{t}&=\left\{\begin{array}{cl}
		\tilde{x}_t    & \text{  if } b_t(\tx_t) \le B_t  \\
		0
		& \text{ otherwise} 
		\end{array} \right. \ , \nonumber \\
		B_{t+1} &= B_t - b_t(x_t) . \nonumber
		\end{align}\\
		
		Determine the target resource consumption:
		\[
		a_t = \arg\max_{a\le \rho}\{r(a)+\mu_t^\top a\}\ .
		\]\\
		
		Obtain a subgradient of $D_t(\mu_t| \gamma_t)$: $$ \  \tg_t = -b_t (x_t) + a_t\ .$$\\

		 Update the dual variable by mirror descent:
			\begin{equation}\label{eq:dual_update}
           \mu_{t+1} = \arg\min_{\mu\in\cD} \langle \tg_t, \mu \rangle + \frac{1}{\eta} V_h(\mu, \mu_t) \ ,
       \end{equation}
    \color{black} where $V_h(x,y)=h(x)-h(y)-\nabla h(y)^\top (x-y)$ is the Bregman divergence.
	}
	\caption{Dual Mirror Descent Algorithm for \eqref{eq:OPT}}
	\label{al:sg}
\end{algorithm}


Algorithm \ref{al:sg} presents the main algorithm we study in this paper. Our algorithm keeps a dual variable $\mu_t \in \cD$ for each resource that is updated using mirror descent~\citep{hazan2016introduction}.

At time $t$, the algorithm receives a request $(f_t,b_t,\cX_t)$, and computes the optimal response $\tx_t$ that maximizes an opportunity cost-adjusted reward of this request based on the current dual solution $\mu_t$ according to equation~\eqref{eq:primal_decision}. When the dual variables are higher, the algorithm naturally takes decisions that consume less resources. It then takes this action (i.e., $x_t=\tx_t$) if the action does not exceed the resource constraint, otherwise it takes a void action (i.e., $x_t=0$). Additionally, it chooses a target resource consumption $a_t$ by maximizing the opportunity-cost adjusted regularized (in the next subsection we give closed-form expressions for $a_t$ for the regularizers we consider).

Writing the dual function as $D(\mu | \vgamma ):=\sum_{t=1}^T D_t(\mu | \gamma_t)$ where the $t$-th term of the dual function is given by $D_t(\mu | \gamma_t) = f_t^*(\mu)  +r^*(-\mu)$, it follows that $\tg_t=-b_t (x_t)+a_t$ is subgradient of $D_t(\mu| \gamma_t)$ at $\mu_t$ whenever decisions are not constrained by resources. To see this, notice that it follows from the definition of conjugate function \eqref{eq:conjugate} that $-b_t(x_t)\in\partial f_t^*(\mu_t)$ and $a_t \in \partial r^*(-\mu_t)$. Therefore
\begin{align*}
&\tg_t= -b_t (x_t)+a_t \in
\partial \left( f_t^*(\mu_t) + r^*(-\mu_t) \right) \in \partial D_t(\mu_t| \gamma_t) \ .
\end{align*}
Finally, the algorithm utilizes $\tg_t$ to update the dual variable by performing the online dual mirror descent descent step in \eqref{eq:dual_update} with step-size $\eta$ and reference function $h$. 

Intuitively, by minimizing the dual function the algorithm aims to construct good dual solutions which, in turn, would lead to good primal performance. The mirror descent step~\eqref{eq:dual_update} updates the dual variable by moving in the opposite direction of the gradient (because we are minimizing the dual problem) while penalizing the movement from the current solution using the Bregman divergence. Dual variables are projected to the set $\cD$ to guarantee dual feasibility. The target resource consumption $a_t$ and the geometry of dual feasible set $\cD$ capture the tension between reward maximization and the regularizer. 

When the optimal dual variables are interior, by minimizing the dual function, the algorithm strives to make subgradients vanishing or, equivalently, set the consumption close to the target resource consumption, i.e., $b(x) = a$. Conversely, when optimal dual variables lie in the boundary of $\cD$, we have that $b(x) \le a$  because the non-negative orthant lies in the recession cone of $\cD$. By definition, we know that $a \le \rho$, which guarantees that the resource constraints are satisfied, i.e., on average at most $\rho$ unit of resources should be spent per time period. The actual value of target resource consumption is dynamically adjusted based on the dual variables and the regularizer. 

In most of our examples, the regularizer plays an active role when the dual variables are low. When the dual variables are higher than the marginal contribution of the regularizer (as captured by its Lipschitz constant), property (iv) of Lemma~\ref{lem:mu_negative} implies that $a_t = \rho$ and decisions are made to satisfy the resource constraints. Moreover, in this case the projection onto $\cD$ does not play a role because dual variables are far from the boundary. This property is intuitive: a high opportunity cost signals that resources are scarce and, as a result, maximizing rewards subject to  resource constraints---while ignoring the regularizer---becomes the primary consideration.

Assumption \ref{ass:h} presents the standard requirements on the reference function $h$ for an online mirror descent algorithm:

\begin{ass}[Assumptions on reference function $h$]\label{ass:h}The reference function satisfies:
 \begin{enumerate}
     \item $h(\mu)$ is either differentiable or essentially smooth in $\cD$.

     \item $h(\mu)$ is $\sigma$-strongly convex in $\ell_1$-norm in $\mathcal{D}$, i.e., $h(\mu_1)\ge h(\mu_2) + \langle \nabla h(\mu_2), \mu_1-\mu_2\rangle + \frac{\sigma}{2}\|\mu_1-\mu_2\|_1^2$ for any $\mu_1,\mu_2\in \mathcal{D}$.
 \end{enumerate}
 \end{ass}

Strong convexity of the reference function is a standard assumption for the analysis of mirror descent algorithms~\citep{bubeck2015convex}. The previous assumptions imply, among other things, that the projection step \eqref{eq:dual_update} of the algorithm always admits a solution. To see this, note that continuity of the regularizer implies that $\cD$ is closed by Proposition 1.1.6 of \cite{bertsekas2009convex}. The objective is continuous and coercive. Therefore, the projection problem admits an optimal solution by Weierstrass theorem. We here use the $\ell_1$ norm to define strong convexity in the dual space (i.e., in $\mu$-space), which is the dual norm of the $\ell_{\infty}$ norm that, in turn, is the norm used in the primal space as per Assumption~\ref{ass:p}. 


Algorithm~\ref{al:sg} only takes an initial dual variable and a step-size as inputs, and is thus simple to implement. In practice, the step-size can be tuned, similar to other online first-order methods.  We analyze the performance of our algorithm in the Section~\ref{sec:performance}. In some cases, though, the constraint set $\cD$ can be complicated, but a proper choice of the reference function may make the descent step \eqref{eq:dual_update} easily computable (in linear time or with a closed-form solution). In the rest of this section, we discuss the constraint sets $\cD$ and suitable references function $h$ for each of the examples.


\subsection{Application of the Example Regularizers}\label{sec:algorithm-examples} 

We now apply Algorithm~\ref{al:sg} to each of the examples presented in Section~\ref{sec:ex}. For each of these examples, we characterize the dual feasible set and the target resource consumption $a^*(-\mu) \in \arg\max_{a\le \rho}\{r(a)+\mu^\top a\}$. Notice that with the introduction of a regularizer, the dual constraint set $\cD$ can become complex. For example, when the regularizer is max-min fairness, the set $\cD$ is a polytope with exponentially many cuts. As a result, the projection step in \eqref{eq:dual_update} sometimes may be difficult. Fortunately, mirror descent allows us to choose the reference function $h$ to capture the geometry of the feasible set, and a suitable reference function can get around such issues. Next, we present different choices of reference functions for each example, such that \eqref{eq:dual_update} has either closed-form solution or it can be solved by a simple optimization program. Proofs of all mathematical statements are available in Appendix~\ref{sec:appendix-examples}.

\paragraph{Example~\ref{ex:no_regularizer} (No Regularizer).} When $r(a) = 0$, we have that $\cD=\RR^m_+$ and, for $\mu\in\cD$, $r^*(-\mu) = \mu^\top \rho$ and $a^*(-\mu) = \rho$. In this case, the algorithm always uses $\rho$ as the target resource consumption, i.e., it aims to deplete $\rho$ resources per unit of time. The projection to the non-negative orthant can be computed in closed form. If we choose the reference function to be the squared Euclidean norm $h(\mu)=\frac{1}{2}\|\mu\|_2^2$, the update \eqref{eq:dual_update} has the closed-form solution $\mu_{t+1}=\max( \mu_t-\eta g_t, 0)$, where the maximum is to be interpreted coordinate-wise. This recovers online subgradient descent. If we choose the reference function to be the negative entropy $h(\mu)=\sum_{j=1}^m \mu_j\log(\mu_j)$, the update \eqref{eq:dual_update} has the closed-form solution $\mu_{t+1}=\mu_t \circ \exp(-\eta g_t)$ where $\circ$ denotes the coordinate-wise product. This recovers the multiplicative weights update algorithm.

\paragraph{Example~\ref{ex:minimal_cons2} (Max-min Fairness).} When $r(a) = \lambda\min_{j} (a_j/\rho_j)$, then $\cD=\big\{\mu \in \mathbb R^m \mid \sum_{j \in S} \rho_j \mu_j \ge -\lambda \ \forall S \subseteq [m]\big\}$, and, for $\mu \in\cD$,  $r^*(-\mu)= \rho^\top \mu+\lambda$ and $a^*(-\mu) = \rho$. As a result, the target resource consumption is always set to $\rho$ and the regularizer impacts the decisions of the algorithm via the feasible region $\cD$. To gain some intuition, note that the feasible set can be alternatively written as $\cD=\big\{\mu \in \mathbb R^m \mid \sum_{j=1}^m \rho_j (\mu_j)^- \le  \lambda \big\}$ because only the constraints for resources with negative dual variables can bind in the original polyhedral representation (this set resembles the lower dotted boundary in Figure~\ref{fig:dual_feasible_region}). Therefore, for any single resource we have that $\mu_j \ge -\lambda/\rho_j$. By letting dual variables go below zero, the algorithm can force the consumption of a resource to increase. The intuition is simple: even if consuming a resource might not beneficial in terms of reward, increasing consumption can increase fairness. The constraint $\sum_{j=1}^m \rho_j (\mu_j)^- \le  \lambda$ bounds the total negative contribution of dual variables and indirectly guarantees that only resources with the lowest relative levels of consumption are prioritized.

Note that there are exponential number of linear constraints in the domain $\cD$. Fortunately, we can get around this issue when the reference function $h$ is chosen to have the form $h(\mu) = \sum_{j=1}^m q(\rho_i \mu_i)$ for some convex function $q:\RR \rightarrow \RR$ by utilizing the fact that $\cD$ is coordinate-wisely symmetric in $\rho \circ \mu$, where $\circ$ is the coordinate-wise product.
	
To obtain $\mu_{t+1}$, we first compute the dual update without projection as
\begin{align}\label{eq:update-wo-projection}
    \tmu_t = \nabla h^*(\nabla h(\mu)-\eta \tg_t)\,,
\end{align}
where $h^*$ is the convex conjugate of the reference function (see, e.g., \citealt[Excercise 3.40]{boyd2004convex}). Then, it holds that
	\begin{equation}\label{eq:subp}
	\mu_{t+1}=\arg\min_{\mu\in\cD} V_h(\mu, \tmu_t)\ .
	\end{equation}
{After rescaling with $\rho$, both the reference function $h(\mu)$ and the domain $\cD$ are coordinate-wisely symmetric. Therefore,} the projection problem \eqref{eq:subp} keeps the order of $\rho \circ\mu$ as $\nabla h$ is monotone.\footnote{This can be easily seen by contradiction as following: Suppose there exists $(j_1)<(j_2)$ such that $\rho_{j_1}(\tmu_t)_{j_1}<\rho_{j_2}(\tmu_t)_{j_2}$ and $\rho_{j_1}(\mu_{t+1})_{j_1}>\rho_{j_2}(\mu_{t+1})_{j_2}$. Consider the solution $\hmu_{t+1}$ which is equal to $\mu_{t+1}$ except on the coordinates $j_1$ and $j_2$, in which we set
	$(\hmu_{t+1})_{j_1}=\frac{\rho_{j_2}}{\rho_{j_1}}{(\mu_{t+1}})_{j_2}$ and $(\hmu_{t+1})_{j_2}=\frac{\rho_{j_1}}{\rho_{j_2}}{(\mu_{t+1})}_{j_1}$, 
	then it holds by the symmetry of $\cD$ in $\rho\circ \mu$ that $\hmu_{t+1}\in\cD$, and moreover, it holds that
\begin{align*}
    V_h(\mu_{t+1}, \tmu_t)-V_h(\hmu_{t+1}, \tmu_t)
    =& h(\mu_{t+1}) - h(\hmu_{t+1}) + \langle \nabla h(\tmu_t), \hmu_{t+1}-\mu_{t+1} \rangle \\ 
    =& \left\langle \dot{q}\pran{\rho_{j_1}(\tmu_t)_{j_1}}-\dot{q}\pran{\rho_{j_2}(\tmu_t)_{j_2}}, \rho_{j_1}(\tmu_t)_{j_1} - \rho_{j_2}(\tmu_t)_{j_2} \right\rangle 
    >0 \ ,
\end{align*}
where the last inequality is due to the strong convexity of $q$. This contradicts with the optimality of $\mu_{t+1}$ as given in equation \eqref{eq:subp}.} Therefore, if we relabel indices so that $\rho\circ\tmu_t$ is sorted in non-decreasing order, i.e., $\rho_{1}(\tmu_t)_{1}\le \rho_{2}(\tmu_t)_{2} \le \dots \le \rho_{m}(\tmu_t)_{m}$ we can reformulate \eqref{eq:subp} as:
	\begin{align*}
	\min_{\mu \in \RR^m}\,     &  V_h(\mu, \tmu_t)=\sum_{j=1}^m V_{q}(\rho_{j} \mu_{j} , \rho_{j} (\tmu_t)_{j} ) \\
	\text{s.t. }    & \sum_{j=1}^s \rho_{j} \mu_{j} \ge -\lambda \text{ \ \ \ \     for \ \ \ \ } s = 1,\ldots,m\,.
	\end{align*}
In particular, when $q(\mu)=\mu^2/2$, the latter problem is a $m$-dimensional convex quadratic programming with $2m$ constraints, which can be efficiently solved by convex optimization solvers.

\paragraph{Example~\ref{ex:minimal_cons} (Load Balancing).} When $r(a) =  \lambda \min_{j} \big( (\rho_j - a_j)/\rho_j \big)$, then $\cD=\big\{\mu \ge 0 \mid \sum_{j=1}^m \rho_j \mu_j \ge \lambda\big\}$, and, for $\mu \in\cD$,  $r^*(-\mu)= \rho^\top \mu-\lambda$ and $a^*(-\mu) = \rho$. As in the case of the fairness regularizer, the target resource consumption is always set to $\rho$ and the regularizer impacts the decisions of the algorithm via the feasible region $\cD$. The feasible region is the non-negative orthant minus a scaled simplex with vertices $\lambda / \rho_j$. Recall that the load balancing regularizer seeks to maximize the minimum relative resource availability.  When the opportunity cost of any single resource is very large, in the sense that $\mu_j > \lambda / \rho_j$ for some $j \in [m]$, resource $j$ is so scarce that increasing its availability is not beneficial. In this case, resource $j$ would end up being fully utilized and, because the regularizer is zero, load balancing no longer plays a role in the algorithm's decision. The load balancing regularizer kicks in when all resources are not-too-scarce. Recall that resources with typical consumption below the target would have dual variables close to zero. Therefore, to balance the load, the dual variables are increased so that the relative consumption of the resources match.

In this case $\cD$ is a polyhedron with a linear number of constraints and \eqref{eq:dual_update} can be computed by solving a simple quadratic program if we employ the squared weighted Euclidean norm $h(\mu)=\frac{1}{2}\|\rho \circ \mu\|_2^2$ as a regularizer. Interestingly, using a scaled negative entropy we can compute the update in closed form. In particular, consider the reference function $h(\mu) = \sum_{j=1}^{m} \rho_j \mu_j \log(\rho_j \mu_j)$. Let $\tilde \mu_{t}$ be the dual update without the projection as in \eqref{eq:update-wo-projection}, which in this case is given by $\tilde \mu_{t} = {\mu_t} \circ \exp(-\eta \tg_t \circ \rho^{-1})$. Then, the dual update is obtained by solving \eqref{eq:subp}, which amounts to projecting back to the scaled unit simplex. For the scaled negative entropy, we have:
 \[
     \mu_{t+1} =
     \begin{cases}
         \tilde \mu_{t} & \text{if } \rho^\top \tmu_{t} \ge \lambda\,,\\
         \frac{\lambda} {\rho^\top \tilde \mu_{t}} \tmu_{t}& \text{otherwise}\,,
     \end{cases}
 \]
and, as a result, the dual update can be performed efficiently without solving an optimization problem.


\paragraph{Example~\ref{ex:hinge_loss} (Below-Target Consumption).} When $r(a)=\sum_{j=1}^m c_j \min(\rho_j-a_j, \rho_j-t_j)$, then $\cD=\mathbb R_+^m$ and, for $\mu \in \cD$, $r^*(-\mu) =  \mu^\top t + \sum_{j=1}^m (\rho_j - t_j) \max(\mu_j, c_j)$ and $a_j^*(-\mu) = t_j$ if $\mu_j \in [0, c_j)$ and $a_j^*(-\mu) = \rho_j$ if $\mu_j\ge c_j$. In this case, the dual feasible set is the non-negative orthant as in the case of no regularization. Therefore, the projection can be conducted as in Example~\ref{ex:no_regularizer}. The regularizer impacts the algorithm's decision by dynamically adjusting the target resource consumption $a_t$. Recall that the purpose of the below-target regularizer is to reduce resource utilization by granting the decision maker a reward $c_j$ for reducing each unit of resource below a target $t_j$. If the opportunity cost of resources is above $c_j$, then reducing resource consumption is not beneficial and the target is set to $a_t = \rho$. Conversely, if the opportunity costs of resources are below $c_j$, then the target is set to $a_t = t_j$ and the algorithm reduces the resource consumption to match the target. Note that reducing the target to $t_j$ leads to smaller subgradients $(\tg_t)_j$, which, in turn, increases future dual variables. This self-correcting feature of the algorithm guarantees that the target is reduced to $t_j$, which can be costly reward-wise, only when economically beneficial.

\paragraph{Example~\ref{ex:hinge_loss2} (Above-Target Consumption).} When $r(a)=\sum_{j=1}^m c_j \min(a_j, t_j)$, then $\cD=\left\{\mu \in \RR^m \mid \mu \ge -c\right\}$ and, for $\mu \in \cD$, $r^*(-\mu) = c^\top t + \mu^\top t + \sum_{j=1}^m (\rho_j - t_j) \max(\mu_j, 0)$ and $a_j^*(-\mu) = t_j$ if $\mu_j \in [-c_j, 0)$ and $a_j^*(-\mu) = \rho_j$ if $\mu_j\ge0$. In this case, the dual feasible region $\cD$ is the shifted non-negative orthant and subgradient descent (i.e., a squared Euclidean norm) gives a closed-form solution to~\eqref{eq:dual_update}. The objective of the above-target regularizer is to increase resource consumption by offering a unit reward of $c_j$ for each unit closer to the target $t_j$. Because the regularizer encourages resource consumption, the dual variables can be negative and, in particular, $\mu_j \ge -c_j$. If the dual variable of a resource is negative, increasing resource consumption beyond the target is not beneficial and the target is set to $t_j$. When the dual variables are positive, the algorithm behaves like the case of no regularization.


\vspace{0.2cm}







\vspace{0.2cm}

\section{Theoretical Results}\label{sec:performance}


We analyze the performance of our algorithm under stochastic and adversarial inputs. We shall see that our algorithm is oblivious to the input model and attains good performance under both inputs without knowing in advance which input model it is facing.

\subsection{Stochastic I.I.D. Input Model}\label{sec:iid-model}
In this section, we assume that requests are independent and identically distributed (i.i.d.) from a probability distribution $\cP\in \Delta(\cS)$ that is unknown to the decision maker, where $\Delta(\cS)$ is the space of all probability distributions over the support set $\cS$. We measure the regret of an algorithm as the worst-case difference over distributions in $\Delta(\cS)$, between the expected performance of the benchmark and the algorithm:
    \begin{align*}
    \Regret{A} = \sup_{\cP \in \Delta(\cS)}  \left\{ \EE_{\gamma_t \sim \cP} \left[ \OPT(\vgamma) - R(A|\vgamma) \right] \right\}\,.
    \end{align*}
We say an algorithm is low regret if the regret grows sublinearly with the number of periods. The next theorem presents the worst-case regret bound of Algorithm~\ref{al:sg}.

	\begin{thm}\label{thm:master}		Consider Algorithm \ref{al:sg} with step-size $\eta \ge 0$ and initial solution $\mu_0\in \cD$. Suppose Assumptions 1-2 are satisfied. Then, it holds for any $T\ge 1$ that
		\begin{align}\label{eq:master}
		\begin{split}
		\Regret{A}\le C_1 + C_2 \eta T + \frac {C_3} {\eta}\,.
		\end{split}
		\end{align}
		where $C_1 = \bar b (\bar f + \ubr + L (\ubb + \uba)) / \lbrho$, $C_2 = (\ubb + \uba)^2/ (2\sigma)$, $C_3 = \sup \Big\{ V_h(\mu, \mu_0) : \mu \in \cD, \|\mu\|_1 \le L + C_1/\bar b  \Big\}$ and $\lbrho = \min_{j \in [m]} \rho_j$.
	\end{thm}

A few comments on Theorem \ref{thm:master}: First, the previous result implies that, by choosing a step-size of order $\eta \sim T^{-1/2}$, Algorithm~\ref{al:sg} attains regret of order $O(T^{1/2})$ when the length of the horizon and the initial amount of resources are simultaneously scaled. 
Second, Lemma~1 from \citet{ArlottoGurvich2019} implies that one cannot hope to attain regret lower than $\Omega(T^{-1/2})$ under our assumptions. (Their result holds for the case of no regularizer, which is a special case of our setting.) Therefore, Algorithm~\ref{al:sg} attains the optimal order of regret. Third, while we prove the result under the assumption that the regularizer is non-negative a similar result holds when the regularizer takes negative values because the regret definition is invariant to the value shift. 



The proof of Theorem~\ref{thm:master} combines ideas from \cite{AgrawalDevanue2015fast} and \cite{balseiro2020best} to incorporate hard resource constraints and regularizers, respectively. We prove the result in three steps. Denote by $\tA$  the stopping time corresponding to the first time that a resource is close to be depleted. We first bound from below the cumulative reward of the algorithm up to $\tA$ in terms of the dual objective minus the complementary slackness term $\sum_{t=1}^{\tA} \mu_t^\top(a_t - b_t(x_t)) = \sum_{t=1}^{\tA} \mu_t^\top g_t$. In the second step, we upper bound the complementary slackness term by noting that, until the stopping time $\tA$, the algorithm's decisions are not constrained by resources and, as a result, it implements online mirror descent on the dual function because $g_t$ are subgradients. Using standard results for online mirror descent, we can compare the complementary slackness term to its value evaluated at a fixed static dual solution, i.e., $\mu^\top \sum_{t=1}^{\tA}g_t$ for any dual feasible vector $\mu \in \cD$. We then relate this term to the value of the regularizer and the performance lost when depleting resources to early, i.e., in cases when $\tA < T$. To do so, a key idea involves decomposing the fixed static dual solution as $\mu = \hat \mu + \delta$ where $\delta \ge 0$ and $\hat \mu \in \cD$ is a carefully picked dual feasible point that is determined by the regularizer $r$ (such decomposition is always possible because the positive orthant is in the recession cone of $\cD$). By using the theory of convex conjugates, we can relate $\hat \mu^\top \sum_{t=1}^{\tA}g_t$ to the value of the regularizer. In the third step, we upper bound the optimal reward under full information in terms of the dual function using \eqref{eq:obtain_dual} and then control the performance lost from stopping early at a time at $\tA<T$ in terms of $\delta^\top \sum_{t=1}^{\tA}g_t$ by choosing $\delta$ based on which resource is depleted.

Theorem~\ref{thm:master} readily implies that Algorithm~\ref{al:sg} attains $O(T^{1/2})$ regret for all examples presented in Section~\ref{sec:ex}. The only exceptions are Example~\ref{ex:no_regularizer} with the negative entropy regularizer and  Example~\ref{ex:minimal_cons} with the scaled negative entropy regularizer, which lead to closed-form solutions for the dual update. The issue is that the negative entropy function does not satisfy Assumption~\ref{ass:h} because its strong convexity constant goes to zero as the dual variables go to infinity. Using a similar analysis to that of Proposition~2 of \cite{balseiro2020best}, we can show that the dual variables remain bounded from above in both examples. By restricting the reference function to a box, we can show that the strong convexity constant of the negative entropy remains bounded from below, which implies that Algorithm~\ref{al:sg} attains sublinear regret.

\subsection{Adversarial Input}\label{sec:adversarial}
    
In this section, we assume that requests are arbitrary and chosen adversarially. Unlike the stochastic i.i.d.~input model, regret can be shown to grow linearly with $T$ and it becomes less meaningful to study the order of regret over $T$. Instead, we say that the algorithm $A$ is asymptotically $\alpha$-competitive, for $\alpha \ge 1$, if
    \[
        \lim\sup_{T\rightarrow\infty} \sup_{\vec \gamma \in \cS^T}  \left\{ \frac 1 T \Big( \OPT(\vec \gamma) - \alpha R(A | \vec \gamma) \Big) \right\}\le 0\,.
    \]
    An asymptotic $\alpha$-competitive algorithm asymptotically guarantees fraction of at least $1/\alpha$ of the best performance in hindsight over all possible inputs.
    


In general, the competitive ratio depends on the structure of the regularizer $r$ chosen. The next theorem presents an analysis for general concave regularizers.

\begin{thm}\label{thm:adversial}
Suppose there exists a finite $\alpha\ge 1$ and $p\in \partial r(0)$, such that for any $\mu\in\cD$ it holds
\begin{equation}\label{eq:adv_condition}
    \sup_{(f,b,\cX) \in \cS} \sup_{x\in\cX} \left\{ (\mu + p)^\top b(x) \right\} + r(0)  \le \alpha r^*(-\mu) \ .
\end{equation}
Consider Algorithm~\ref{al:sg} with step-size $\eta$. Then, it holds for any $T\ge 1$ that
		\begin{align*}
		\begin{split}
		\OPT(\vec \gamma) - \alpha R(A | \vec \gamma) \le
		C_1 + C_2 \eta T + \frac {C_3} {\eta} \ .
		\end{split}
		\end{align*}
		where $C_1 = \ubb (\ubf+\ubr+\alpha L (\ubb +\uba)+ 2\alpha \|p\|_1\ubb)/\ubrho$, $C_2 = \alpha (\ubb + \uba)^2/(2\sigma)$, and $C_3 = \alpha \sup \Big\{ V_h(\mu, \mu_0) : \mu \in \cD, \|\mu\|_1 \le L + C_1/(\alpha \ubb)   \Big\}$.
\end{thm}

Theorem~\ref{thm:adversial} implies that when the step-size is of order $\eta \sim T^{-1/2}$, Algorithm~\ref{al:sg} is asymptotic $\alpha$-competitive. Condition \eqref{eq:adv_condition} can be shown to hold with finite $\alpha$ if $r(0)>0$ or if $r(0)=0$ and $r$ is locally linear around zero. We next discuss the value of the competitive ratio for the examples presented in Section~\ref{sec:ex}. Proofs of all mathematical statements are available in Appendix~\ref{sec:cr-examples}.

\paragraph{Example~\ref{ex:no_regularizer} (No Regularizer).} When $r(a) = 0$, then Algorithm \ref{al:sg} is $\max\{\alpha_0,1\}$-competitive with $\alpha_0 = \sup_{(f, b, \cX) \in \cS} \sup_{j \in [m], x\in\cX} b_j(x) / \rho_j$. This recovers a result in \cite{balseiro2020best}, which is unimprovable without further assumptions on the input. The competitive ratio $\alpha_0$ captures the relative scarcity of resources by comparing, for each resource, the worst-case consumption of any request as given by $\sup_{x\in\cX} b_j(x)$ to the average resource availability $\rho_j = B_j / T$. Naturally, the competitive ratio degrades as resources become scarcer.

\paragraph{Example~\ref{ex:minimal_cons} (Load Balancing).} When $r(a) = \lambda \min_{j \in [m]} \big((\rho_j - a_j)/\rho_j \big)$, then Algorithm \ref{al:sg} is $(\alpha_0 + 1)$-competitive. Notably, introducing a load balancing regularized only increases the competitive ratio by one and the competitive ratio does not depend on the tradeoff parameter $\lambda$. 

\paragraph{Examples~\ref{ex:hinge_loss} and~\ref{ex:hinge_loss2} (Below- and Above-Target Consumption).} In this case, the competitive ratio is $\alpha = \max\{\sup_{(f, b, \cX) \in \cS} \sup_{j \in [m], x\in\cX} b_j(x) / t_j, 1\}$. The competitive ratio is similar to that of no regularization with the exception that resource scarcity is measured with respect to the targets $t$ instead of the resource constraint vector $\rho$. Because $t \le \rho$, the introduction of these regularizers deteriorates the competitive ratios relative to the case of no regularizer. Interestingly, the competitive ratios do not depend on the unit rewards $c$ in the regularizer.


\paragraph{Example 2 (Max-min Fairness)} Unfortunately, Theorem~\ref{sec:adversarial} does not lead to a finite competitive ratio for the max-min fairness regularizer. For this particular regularizer, we can provide an ad hoc competitive ratio analysis from first principles.

\begin{thm}\label{thm:adv-fairness}
Consider Algorithm \ref{al:sg} with step-size $\eta$ for regularized online allocation problem with max-min fairness (i.e., $r(a) = \lambda \min_{j} (a_j/\rho_j)$). Then:
\begin{enumerate}
    \item If the reward function $f_t(x)$ satisfies Assumption \ref{ass:p}, the consumption function is $b_t(x)=x$, and the action space is the unit simplex $\cX_t=\{ x\in [0,1]^m : \sum_{j=1}^n x_j \le 1 \}$, then Algorithm~\ref{al:sg} is $\pran{\max_j \left\{1/\rho_j\right\}+ 1}$-competitive.
    \item If the reward function $f_t(x)$, the consumption function $b_t(x)$, and the action space $\cX_t$ satisfy Assumption~\ref{ass:p}, then Algorithm~\ref{al:sg} is $\pran{\max\{\beta_1, \beta_2\}+1}$-competitive, where $\beta_1, \beta_2>0$ satisfy $\rho\in\beta_1 b(\cX)$ and $b( x) \le \beta_2 \rho$ for every request $(f,b,\cX) \in \cS$.
\end{enumerate}
\end{thm}

The first part of Theorem~\ref{thm:adv-fairness} provides an analysis for online matching problems in which the objective is to maximize reward plus the max-min fairness regularizer. Online matching (without regularization) is a widely studied problem in the computer science and operations research literature (see, e.g., \citealt{karp1990optimal,mehta2013online}). In online matching, the decision maker needs to match the arriving request to one of $m$ resources and each assignment consumes one unit of resource. The reward function $f_t(x)$ is arbitrary. Here, we consider a variation in which the objective involves maximizing the revenue together with the max-min fairness of the allocation. Notice that in the case of no regularization, the competitive ratio for this problem is $\alpha_0 = \max_j 1/\rho_j$. Therefore, the introduction of the max-min fairness regularizer increases the competitive ratio by one and, as before, the competitive ratio is independent of the tradeoff parameter $\lambda$. 

The second part of the theorem extends the result to generic online allocation problems with the max-min fairness regularizer. The parameter $\beta_2$ is similar to $\alpha_0$ in that it compares the worst-case resource consumption of any request to the average consumption. The parameter $\beta_1$ is more subtle and measures the distance of the resource constraint vector $\rho$ to the set of achievable resource consumption $b(\cX) = \{b(x)\}_{x \in \cX}$ akin to the Minkoswki functional. That is, it captures how much we need to scale the set $\bigcap_{(f,b,\cX)\in \cS} b(\cX)$ for $\rho$ to lie in it. The parameter $\beta_1$ is guaranteed to be bounded whenever $b_j(x) > 0$ for every request. Unfortunately, if $b_j(x) = 0$ for some request, Theorem~\ref{thm:adv-fairness} does not provide a bounded competitive ratio. Intuitively, in such cases, the adversary can make max-min fairness large by first offering requests that consume all resources and then switching to requests that do not consume resources that were not allocated in the first stage. We conjecture that when the conditions of Theorem~\ref{thm:adv-fairness} do not hold, no algorithm that attains vanishing regret for stochastic input can achieve bounded competitive ratios in the adversarial case. 

\section{Numerical Experiments}
In this section, we present numerical experiments on a display advertisement allocation application regularized by max-min fairness on consumption (Example \ref{ex:minimal_cons2}).

\textbf{Dataset.} We utilize the display advertisement dataset introduced in \cite{balseiro2014yield}. They consider a publisher who has agreed to deliver ad slots (the requests) to different advertisers (the resources) so as to maximize the cumulative click-through rates (the reward) of the assignment. In their paper, they estimate click-through rates using mixtures of log-normal distributions. We adopt their parametric model as a generative model and sample requests from their estimated distributions. We consider publisher 2 from their dataset, which has $m=12$ advertisers. Furthermore, in our experiments, we rescale the budget $\rho$ so that $\sum_{j=1}^m \rho_j=1.5$ in order to make sure that the max-min fairness is strictly less than $1$. This implies that the maximal fairness value is $2/3$.

\textbf{Regularized Online Problem.} The goal here is to design an online allocation algorithm that maximizes the total expected click-through rate with a max-min fairness regularizer on resource consumption as in Example \ref{ex:minimal_cons2}. Advertiser $j \in [m]$ can be assigned at most $T \rho_j$ ad slots and the decision variables lie in the simplex $\cX = \{ x \in \RR_+^m : \sum_{j=1}^m x_j \le 1\}$. Denoting by $q_t \in \RR^m$ the click-through rate of the $t$-th ad slot $T$, we have that the benchmark is given by:
\begin{align}\label{eq:poi_numerical}
\begin{split}
\ \ \max_{x : x_t \in \cX_t } &  \sum_{t=1}^T q_t^\top x_t + \lambda \min_{j=1,\ldots,m} \pran{\sum_{t=1}^T (x_t)_j / \rho_j} \\
\text{s.t.} & \sum_{t=1}^T x_t  \le  T\rho\ ,
\end{split}
\end{align}
where $\lambda$ is the weight of the regularizer. In the experiments, we consider the regularization levels $\lambda\in\{0, 0.1, 0.01, 0.001, 0.0001\}$ and lengths of horizon $T\in \{10^2, 10^3, 2 \cdot 10^3, \ldots, 10^4\}$.

\textbf{Implementation Details.}  In the numerical experiments, we implemented Algorithm~\ref{al:sg} with weights $w_j=\rho_j^2$ and step-size $0.01 \cdot T^{-1/2}$. The dual update \eqref{eq:dual_update} is computed by solving a convex quadratic program as stated in Section \ref{sec:algorithm} using \emph{cvxpy}~\citep{diamond2016cvxpy}. For each regularization level $\lambda$ and time horizon $T$, we randomly choose $T$ samples from their dataset that are fed to Algorithm~\ref{al:sg} sequentially. In order to compute the regret, we utilize the dual objective evaluated at the average dual $D(\frac{1}{T}\sum_{t=1}^T \mu_t)$ as an upper bound to the benchmark. We report the average cumulative reward, the average max-min consumption fairness, and the average regret of $100$ independent trials in Figure~\ref{fig:numerical}. Ellipsoids in Figure~\ref{fig:numerical}~(b) give 95\% confidence regions for the point estimates.

\begin{figure}
	\begin{subfigure}{.49\textwidth}
		\centering
		\includegraphics[width=0.95\textwidth]{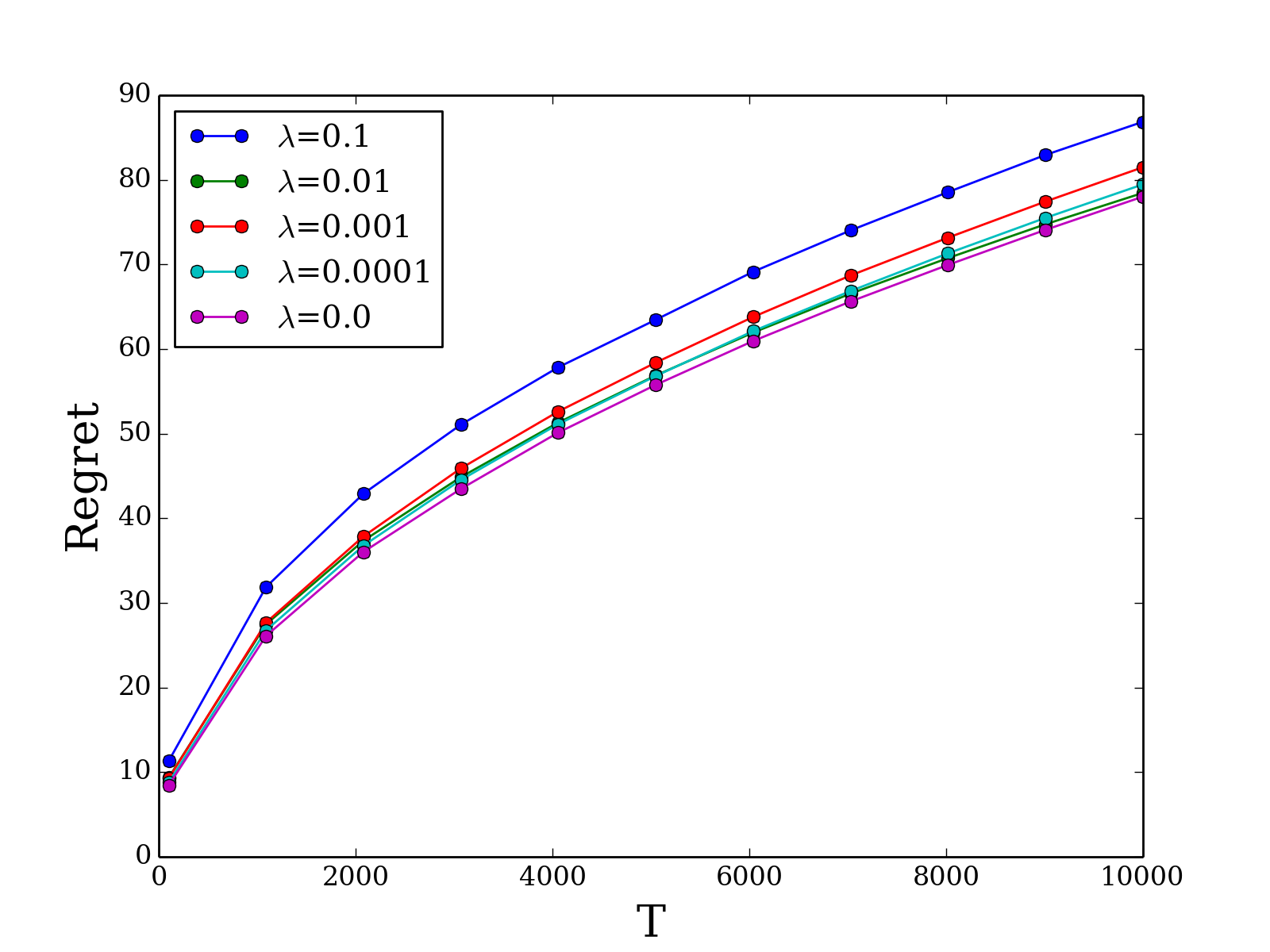}
	\end{subfigure}
	\begin{subfigure}{.49\textwidth}
		\centering
		\includegraphics[width=0.95\textwidth]{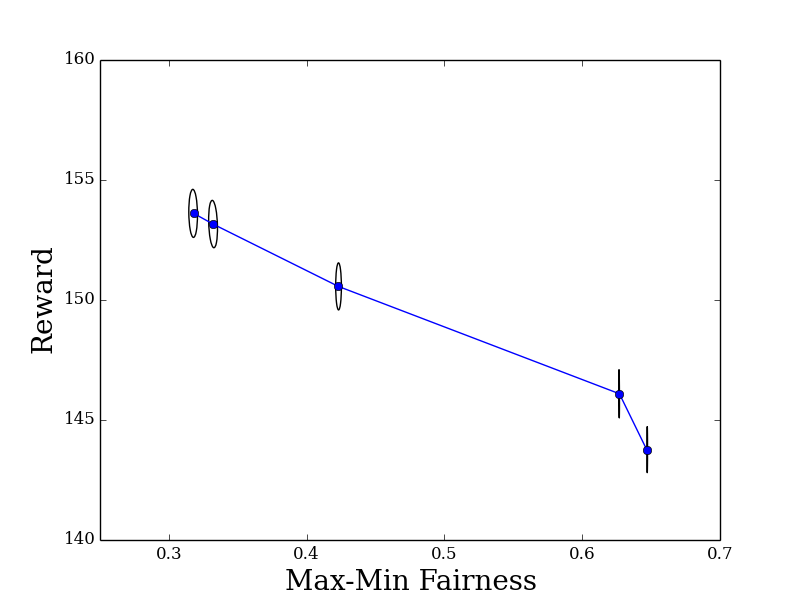}
	\end{subfigure}
	\caption{\small(a) Plot of the regret versus the length of horizon $T$ for the regularized online allocation problem \eqref{eq:poi_numerical} with different regularization levels. (b) Plot of the reward  $\sum_{t=1}^T q_t x_t$ versus the max-min fairness $\min_{j=1,\ldots,m} \pran{\sum_{t=1}^T (x_t)_j / T \rho_j}$. Dots from left to right corresponds to regularization levels $\lambda=0.0, 0.0001, 0.001, 0.01, 0.1$, respectively.}
	\label{fig:numerical}
\end{figure}

\textbf{Discussion.}
Consistent with Theorem \ref{thm:master}, Figure \ref{fig:numerical} (a) suggests that regret grows at rate $O(\sqrt{T})$ for all regularization levels. Figure \ref{fig:numerical} (b) presents the trade-off between reward and fairness with the 95\% confidence ellipsoid. In particular, we can double the max-min fairness while sacrificing only about $4\%$ of the reward by choosing $\lambda=0.01$. This showcases that fairness can be significantly improved (in particular given that the highest possible fairness is $2/3$) by solving the regularized problem with a small amount of reward reduction. 

\section{Conclusion and Future Directions}
In this paper, we introduce the \emph{regularized online allocation problem}, a novel variant of the online allocation problem that allows for regularization on resource consumption. We present multiple examples to showcase how the regularizer can help attain desirable properties, such as fairness and load balancing, and present a dual online mirror descent algorithm for solving this problem with low regret. The introduction of a regularizer impacts the geometry of the dual feasible region. By suitably choosing the reference function of mirror descent, we can adapt the algorithm to the structure of the dual feasible region and, in many cases, obtain dual updates in closed form.

Future directions include extending the results in this work to other inputs such as non-stationary stochastic models. Another direction is to consider other regularizers and develop corresponding adversarial guarantees using our framework. Finally, it is worth exploring whether the competitive ratios we provide are tight (either unimprovable by our algorithm or by any other algorithm) for the regularizers we explore in this paper. 

\bibliographystyle{plainnat}
{\small{\bibliography{references,Lu-papers}}}

\newpage
\appendix
\setstretch{1}

\section{Proofs in Section~\ref{sec:dual-problem}}
\subsection{Proof of Lemma~\ref{lem:mu_negative}}
\begin{proof}

We prove each part at a time.

\textbf{Part (i).} Convexity trivially follows because dual problems are always convex (see, e.g., Proposition 4.1.1 of \citealt{bertsekas2009convex}). Suppose $\mu\in \cD$, namely $\max_{a\le \rho} \{r(a)+\mu^\top a\} < +\infty$. Then it holds for any $e\in \RR^d_+$ and $\lambda>0$ that
\begin{equation*}
    \max_{a\le \rho} \{r(a)+(\mu+\lambda e)^\top a\} \le \max_{a\le \rho} \{r(a)+\mu^\top a\} + \lambda e^\top \rho < +\infty \ ,
\end{equation*}
thus $\mu+\lambda e \in \cD$, which finishes the proof by definition of recession cone.

\textbf{Part (ii).} With $\mu=L e$, it holds for any $a\le \rho$ that
$$
r(a)+ \mu^\top a \le r(\rho) + \mu^\top \rho + L\|\rho-a\|_{\infty} - L e^\top (\rho-a) \le r(\rho) + \mu^\top \rho \ ,
$$
where the first inequality follows from Lipschitz continuity and the second from $\|x\|_\infty \le \|x\|_1$ for every vector $x$. Thus,  $\arg\min_{a\in\cD} r(a)+ \mu^\top a=\rho$, which finishes the proof by definition of $\cD$. \qed

\textbf{Part (iii).} Suppose there is an $\mu\in\cD$, such that $\sum_{j}(\mu_j)^- > L$. Fix $\lambda > 0$. Define $a$ such that $a_j=-\lambda$ for every $j\in\{j:\mu_j\le 0\}$ and $a_j=0$ otherwise. Then it holds that 
$$
r(a)+\mu^\top a \ge r(0)-L\|a\|_{\infty} + \mu^\top a = r(0)+ \lambda\left(\sum_{j}(\mu_j)^- -L\right) \ ,
$$
which goes to $+\infty$ when $\lambda\rightarrow +\infty$ by noticing $\sum_{j}(\mu_j)^- > L$. Thus $\mu \not\in \cD$, which finishes the proof by contradiction. 

\textbf{Part (iv).} We prove the result by contradiction. Let $a \le \rho$ be an optimal solution with $a_j < \rho_j$. Consider the solution $\tilde a$ with $a_j = \rho_j$ and $\tilde a_i = a_i$ for $i\neq j$. Because these two solutions only differ in the $j$-th component, we obtain using Lipschitz continuity
$$
r(a)+ \mu^\top a \le r(\tilde a) + \mu^\top \tilde a + L\|\tilde a-a\|_{\infty} - \mu_j (\rho_j-a_j) < r(\tilde a) + \mu^\top \tilde a \,
$$
where the second inequality follows because $\mu_j > L$ and $\| \tilde a - a\|_\infty = \rho_j - a_j$. Therefore, $a$ cannot be an optimal.
\end{proof}

\section{Proofs in Section~\ref{sec:algorithm}}\label{sec:appendix-examples}

Proposition~\ref{prop:examples} presents the conjugate functions $r^*$, the corresponding domain $\cD$, and optimal actions $a^*(-\mu) \in \arg\max_{a\le \rho}\{r(a)+\mu^\top a\}$ for each example stated in Section \ref{sec:ex}.

\begin{prop}\label{prop:examples}
	The following hold: 
	\begin{itemize}
		\item \textbf{Example \ref{ex:no_regularizer}}: If $r(a) = 0$, then $\cD=\RR^m_+$ and, for $\mu\in\cD$, $r^*(-\mu) = \mu^\top \rho$ and $a^*(-\mu) = \rho$.
		
		\item \textbf{Example \ref{ex:minimal_cons2}}: If $r(a) = \lambda\min_{j} (a_j/\rho_j)$, then $\cD=\big\{\mu \in \mathbb R^m \mid \sum_{j \in S} \rho_j \mu_j \ge -\lambda \ \forall S \subseteq [m]\big\}$, and, for $\mu \in\cD$,  $r^*(-\mu)= \rho^\top \mu+\lambda$ and $a^*(-\mu) = \rho$.
		
		
		\item \textbf{Example \ref{ex:minimal_cons}}: If $r(a) =\lambda \pran{1-\max_{j} (a_j/\rho_j)}$, then $\cD=\big\{\mu \ge 0 \mid \sum_{j=1}^m \rho_j \mu_j \ge \lambda\big\}$, and, for $\mu \in\cD$,  $r^*(-\mu)= \rho^\top \mu-\lambda$ and $a^*(-\mu) = \rho$.        
		
		\item \textbf{Example \ref{ex:hinge_loss}}:
		If $r(a)=\sum_{j=1}^m c_j \min(\rho_j-a_j, \rho_j-t_j)$, then $\cD=\mathbb R_+^m$ and, for $\mu \in \cD$, $r^*(-\mu) =  \mu^\top t + \sum_{j=1}^m (\rho_j - t_j) \max(\mu_j, c_j)$ and $a_j^*(-\mu) = t_j$ if $\mu_j \in [0, c_j)$ and $a_j^*(-\mu) = \rho_j$ if $\mu_j\ge c_j$.
		 
		\item \textbf{Example \ref{ex:hinge_loss2}}: If $r(a)=\sum_{j=1}^m c_j \min(a_j, t_j)$, then $\cD=\left\{\mu \in \RR^m \mid \mu \ge -c\right\}$ and, for $\mu \in \cD$, $r^*(-\mu) = c^\top t + \mu^\top t + \sum_{j=1}^m (\rho_j - t_j) \max(\mu_j, 0)$ and $a_j^*(-\mu) = t_j$ if $\mu_j \in [-c_j, 0)$ and $a_j^*(-\mu) = \rho_j$ if $\mu_j\ge0$.
		
	\end{itemize}
\end{prop}

\subsection{Proof of Proposition~\ref{prop:examples}}

We prove Proposition \ref{prop:examples} for each example at a time.

\textbf{Example~\ref{ex:minimal_cons2}} 

Performing the change of variables $z_j = \lambda( a_j / \rho_j - 1)$ or $a_j = (z_j / \lambda + 1)\rho_j$ we obtain that
\[
r^*(-\mu) = \sup_{a \le \rho} \left\{ \lambda \min\left( \frac{a_j} {\rho_j} \right) + \mu^\top a \right\}
=  \mu^\top \rho + \lambda + \sup_{z \le 0} \left\{  \min\left( z_j \right) + \sum_{j=1}^m \frac{\mu_j \rho_j} \lambda z_j \right\}
\]
and the result follows from invoking the following lemma.

\begin{lem}\label{lemma:minimum1} Let $s(z) = \min_j z_j$ and $s^*(\mu) = \sup_{z \le 0}\{ s(z) + z^\top \mu\}$ for $\mu \in \mathbb R^m$. If $\sum_{j \in S} \mu_j \ge -1$ for all subsets $S \subseteq [m]$, then $s^*(\mu) = 0$ and $z=0$ is an optimal solution. Otherwise, $s^*(\mu) = \infty$.
\end{lem}

\begin{proof}
	Let $\cD = \left\{ \mu \in \RR^m \mid \sum_{j \in S} \mu_j \ge -1 \ \forall S \subseteq [m] \right\}$. We first show that $s^*(\mu) = \infty$ for $\mu \not\in \mathcal D$. Suppose that there exists a subset $S \subseteq [m]$ such that $\sum_{j \in S} \mu_j < -1$. For $t\ge0$, consider a feasible solution with $z_j = -t$ for $j\in S$ and $z_j = 0$ otherwise. Then, because such solution is feasible and $s(z) = -t$ we obtain that $s^*(\mu) \ge s (z ) - t \sum_{j \in S} \mu_j = - t ( \sum_{j \in S} \mu_j + 1)$. Letting $t \rightarrow \infty$, we obtain that $s^*(\mu) = \infty$.
	
	We next show that $s^*(\mu) = 0$ for $\mu \in \mathcal D$. Note that $s^*(\mu) \ge 0$ because $z = 0$ is feasible and $s(0) = 0$. We next show that $s^*(\mu) \le 0$. Let $z \le 0$ be any feasible solution and assume, without loss of generality, that the vector $z$ is sorted in increasing order, i.e., $z_1 \le z_2 \le \ldots \le z_m$. Let $z_{m+1} := 0$. The objective value is
	\begin{align*}
	s(z) + z^\top \mu = z_1 + \sum_{j=1}^m z_j \mu_j
	= \sum_{j=1}^m (z_j - z_{j+1}) \left( 1 + \sum_{i = 1}^j \mu_i \right)\le 0 
	\end{align*}
	where the second equation follows from rearranging the sum and the inequality follows because $z$ is increasing and $\sum_{j \in S} \mu_j + 1 \ge 0$ for all $S \subseteq [m]$. The result follows.
\end{proof}

\textbf{Example~\ref{ex:minimal_cons}}

Performing the change of variables $z_j = \lambda( 1 - a_j / \rho_j)$ or $a_j = (1 - z_j / \lambda)\rho_j$ we obtain that
\[
r^*(-\mu) = \sup_{a \le \rho} \left\{ - \lambda \max\left( \frac{a_j} {\rho_j} \right) + \mu^\top a \right\}
=  \mu^\top \rho + \lambda + \sup_{z \ge 0} \left\{  \min\left( z_j \right) - \sum_{j=1}^m \frac{\mu_j \rho_j} \lambda z_j \right\}
\]
and the result follows from invoking the following lemma.

\begin{lem}\label{lemma:minimum2} Let $s(z) = \min_j z_j$ and $s^*(\mu) = \sup_{z \ge 0}\{ s(z) - z^\top \mu\}$ for $\mu \in \mathbb R^m$. If $\sum_{j=1}^m \mu_j \ge 1$ and $\mu_j \ge 0$ for all $j \in [m]$, then $s^*(\mu) = 0$ and $z=0$ is an optimal solution. Otherwise, $s^*(\mu) = \infty$.
\end{lem}

\begin{proof}
	Let $\cD = \left\{ \mu \in \RR^m \mid \mu_j \ge 0 \text{ and } \sum_{j=1}^m \mu_j \ge 1 \right\}$. We first show that $s^*(\mu) = \infty$ for $\mu \not\in \mathcal D$. First, suppose $\mu_j < 0$ for some $j$. Consider the feasible solution $z = t e_j$ for $t\ge 0$, where $e_j$ is the unit vector with a one in component $j$ and zero otherwise. Then, because such solution is feasible and $s(t e_j) = 0$ we obtain that $ s^*(\mu) \ge s(t e_j) - t \mu_j = -t \mu_j$. Letting $t \rightarrow \infty$, we obtain that $s^*(\mu) = \infty$. Second, suppose that $\sum_{j=1}^m \mu_j < 1$. For $t\ge0$, consider a feasible solution with $z = t e$ where $e$ is the all-one vector. Then, because such solution is feasible and $s(z) =  t$ we obtain that $s^*(\mu) \ge s (z ) - t \sum_{j=1}^m \mu_j = t ( 1 -  \sum_{j=1}^m \mu_j)$. Letting $t \rightarrow \infty$, we obtain that $s^*(\mu) = \infty$.
	
	We next show that $s^*(\mu) = 0$ for $\mu \in \mathcal D$. Note that $s^*(\mu) \ge 0$ because $z = 0$ is feasible and $s(0) = 0$. We next show that $s^*(\mu) \le 0$. Let $z \le 0$ be any feasible solution and assume, without loss of generality, that the vector $z$ is sorted in increasing order, i.e., $z_1 \le z_2 \le \ldots \le z_m$. The objective value is
	\begin{align*}
	s(z) - z^\top \mu = z_1 - \sum_{j=1}^m z_j \mu_j
	\le z_1 \left( 1 - \sum_{j=1}^m \mu_j \right)\le 0 
	\end{align*}
	where the first inequality follows because $z_j \ge z_1$ for all $j \in [m]$ since $z$ is increasing and $\mu_j \ge 0$, and the last inequality follows because $z_1 \ge 0$ and $\sum_{j=1}^m \mu_j \ge 1$. The result follows.
\end{proof}

\textbf{Example~\ref{ex:hinge_loss}}

Notice that
\begin{align*}
    r(a)&=\sum_{j=1}^m c_j \min(\rho_j-a_j, \rho_j-t_j) \\
    &=\sum_{j=1}^m -c_j \max(a_j-\rho_j, t_j-\rho_j) \\
    &=\sum_{j=1}^m c_j (\rho_j-t_j) - c_j \max(a_j-t_j,0) \ .
\end{align*}

Because the conjugate of the sum of independent functions is the sum of the conjugates, we have by Lemma~\ref{lemma:hinge1} that $r^*(-\mu) =  \sum_{j=1}^m \mu^\top t + \sum_{j=1}^m (\rho_j - t_j) \max(\mu_j, c_j)$ for $\mu \ge 0$ and $r^*(-\mu) = \infty$ otherwise.

\begin{lem}\label{lemma:hinge1} Let $r(a) = - c \max(a - t, 0)$ for $a \in \mathbb R$, $c \ge 0$ and $t \in [0,\rho]$. Let $r^*(-\mu) = \sup_{a \le \rho}\{ r(a) + a \mu\}$. Then, $r^*(-\mu) = \mu  t + (\rho - t) \max(\mu - c, 0)$ for $\mu \ge 0$, and $r^*(-\mu) = \infty$ for $\mu < 0$. Moreover, for $\mu \in [0,c]$, $a = t$ is an optimal solution; while, for $\mu \ge c$, $a = \rho$ is an optimal solution.
\end{lem}

\begin{proof}
	We can rewrite the conjugate as
	\[
	r^*(-\mu) = \sup_{a \le \rho}\left\{ - c \max(a - t, 0) + a \mu \right\}
	= \mu t - \inf_{z \le \rho - t} \left\{ c \max(z,0) - \mu z\right\} = \mu t - s(\mu)
	\]
	where the second equation follows by performing the change of variables $ z = a - t$ and the last from setting $s(\mu) := \inf_{z \le \rho - t} \left\{ c \max(z,0) - \mu z\right\}$.
	
	First, suppose that $\mu < 0$. For any $z\le 0$, we have that $s(\mu) \le - \mu z$. Letting $z \rightarrow \infty$ yields that $s(\mu) = -\infty$. Second, consider $\mu \in [0, c]$. Write $s(\mu) = \inf_{z \le \rho - t} \max((c-\mu), z, -\mu z)$. The objective is increasing for $z\ge0$ and decreasing for $z \le 0$. Therefore, the optimal solution is $z = 0$ and $s(\mu) = 0$. Thirdly, for $\mu > c$ a similar argument shows that the objective is decreasing in $z$. Therefore, it is optimal to set $z = \rho - t$, which yields $s(\mu) = - (\rho - t) (\mu - c)$. The result follows from combining the last two cases.
\end{proof}

\textbf{Example \ref{ex:hinge_loss2}}

We have 
\begin{align*}
    r(a)&=\sum_{j=1}^m c_j \min(a_j, t_j) \\
    &=\sum_{j=1}^m -c_j \max(-a_j, -t_j) \\
    &= \sum_{j=1}^m c_j t_j - c_j \max(t_j-a_j, 0) \ .
\end{align*}

Because the conjugate of the sum of independent functions is the sum of the conjugates, we have by Lemma~\ref{lemma:hinge2} that $r^*(-\mu) = c^\top t + \mu^\top t + \sum_{j=1}^m (\rho_j - t_j) \max(\mu_j, 0)$ for $\mu \ge -c$ and $r^*(-\mu) = \infty$ otherwise.

\begin{lem}\label{lemma:hinge2} Let $r(a) = - c \max(t - a, 0)$ for $a \in \mathbb R$, $c \ge 0$ and $t \in [0,\rho]$. Let $r^*(-\mu) = \sup_{a \le \rho}\{ r(a) + a \mu\}$. Then, $r^*(-\mu) = \mu  t + (\rho - t) \max(\mu, 0)$ for $\mu \ge -c$, and $r^*(-\mu) = \infty$ otherwise. Moreover, for $\mu \in [-c,0]$, $a = t$ is an optimal solution; while, for $\mu \ge 0$, $a = \rho$ is an optimal solution.
\end{lem}

\begin{proof}
	We can rewrite the conjugate as
	\[
	r^*(-\mu) = \sup_{a \le \rho}\left\{ - c \max(t - a, 0) + a \mu \right\}
	= \mu t + \sup_{z \le \rho - t} \left\{ c \min(z,0) + \mu z\right\} = \mu t + s(\mu)
	\]
	where the second equation follows by performing the change of variables $ z = a - t$ and the last from setting $s(\mu) := \sup_{z \le \rho - t} \left\{ c \min(z,0) + \mu z\right\}$.
	
	First, suppose that $\mu < -c$. For any $z\le 0$, we have that $s(\mu) \ge (\mu + c) z$. Letting $z \rightarrow -\infty$ yields that $s(\mu) = \infty$. Second, consider $\mu \in [-c, 0]$. Write $s(\mu) = \sup_{z \le \rho - t} \min((c+\mu) z, \mu z)$. The objective is decreasing for $z\ge0$ and increasing for $z \le 0$. Therefore, the optimal solution is $z = 0$ and $s(\mu) = 0$. Thirdly, for $\mu > 0$ a similar argument shows that the objective is increasing in $z$. Therefore, it is optimal to set $z = \rho - t$, which yields $s(\mu) = (\rho - t) \mu$. The result follows from combining the last two cases.
\end{proof}

\section{Proofs in Section \ref{sec:performance}}
\subsection{Proof of Theorem \ref{thm:master}}
{Theorem \ref{thm:master} is an extension of Theorem 1 in \cite{balseiro2020best} to include the concave non-separable regularizer $r$. The proof of Theorem 1 follows a similar proof structure of \cite{balseiro2020best}. In order to deal with the regularizer, the major differences of our proofs compared with that of \cite{balseiro2020best} include: (i) we involve the regularizer $r$ and its conjugate $r^*$, and utilize them throughout the proof; (ii) the choice of the reference value $\hat{\mu}$ (see step 2 in the proof for more details) is chosen critically to handle the existence of the regularizer.} 

We prove the result in three steps. First, we lower bound the cumulative reward of the algorithm up to the first time that a resource is depleted in terms of the dual objective and complementary slackness. Second, we bound the complementary slackness term by picking a suitable ``pivot'' for online mirror descent. We conclude by putting it all together in step three.

\paragraph{Step 1 (Primal performance.)} First, we define the stopping time $\tA$ of Algorithm~\ref{al:sg} as the first time less than $T$ that there exists resource $j$ such that 
\begin{equation*}
\sum_{t=1}^{\tA} b_t(x_t)_j + {\ubb} \ge \rho_j T \ .
\end{equation*}
Notice that $\tau_A$ is a random variable, and moreover, we will not violate the resource constraints before the stopping time $\tau_A$. We here study the primal-dual gap until the stopping time $\tA$. Notice that before the stopping time $\tA$, Algorithm \ref{al:sg} performs the standard subgradient descent steps on the dual function because $\tx_t = x_t$.

Let us denote the random variable $\gamma_t$ to be the type of request in time period $t$, i.e., $\gamma_t$ is the random variable that determines the (stochastic) sample $(f_t,b_t)$ in the $t$-th iteration of Algorithm \ref{al:sg}. We denote $\xi_t=\{\gamma_0,\ldots, \gamma_t\}$.

Consider a time $t \le \tA$ so that actions are not constrained by resources. Because $x_{t} = \arg\max_{x\in\cX_t}\{f_{t}(x)-\mu_{t}^{\top} b_{t} (x)\}$, we have that
\[
f_t(x_t) = f_t^*(\mu_t) + \mu_{t}^{\top} b_{t}  (x_t)\,.
\]
Similarly, because $a_t = \arg\max_{a\le \rho}\{r(a)+\mu_t^\top a\}$, we have that
\[
r(a_t) = r^*(-\mu_t) - \mu_t^\top a_t\,.
\]
{Let $\bar D(\mu | \cP) = \frac 1 T \mathbb E_{\gamma_t \sim \cP}\left[D(\mu | \vgamma)\right] = \EE_{(f,b) \sim \cP}\left[f^*(\mu_t)\right] + r^*(-\mu_t)$ be the expected dual objective at $\mu$ when requests are drawn i.i.d.~from $\cP \in \Delta(\cS)$.} Adding the last two equations and taking expectations conditional on $\sigma(\xi_{t-1})$ we obtain, because $\mu_t \in \sigma(\xi_{t-1})$ and $(f_t,b_t) \sim \cP$, that
\begin{align}\label{eq:bound_one_period}
&\mathbb E\left[ f_t(x_t) + r(a_t) | \sigma(\xi_{t-1}) \right] \nonumber \\
=&\EE_{(f,b) \sim \cP}\left[f^*(\mu_t)\right] + r^*(-\mu_t) + \mu_t^\top \left( \mathbb E\left[ b_{t} (x_t)| \sigma(\xi_{t-1}) \right] - a_t \right)    \nonumber\\
= &  \bar D(\mu_t| \cP) - \mathbb E\left[ \mu_t^\top \left( a_t - b_{t} (x_t) \right) | \sigma(\xi_{t-1}) \right]
\end{align}
where the second equality follows the definition of the dual function. 

Consider the process $Z_t = \sum_{s=1}^t \mu_s^\top \left(a_s - b_{s} x_s\right) - \mathbb E\left[ \mu_s^\top \left(a_s - b_{s} x_s \right) | \sigma(\xi_{s-1}) \right]$, which is martingale with respect to $\xi_t$ (i.e., $Z_t \in \sigma(\xi_t)$ and $\EE[Z_{t+1} | \sigma(\xi_t)] = Z_t$). Since $\tA$ is a stopping time with respect to $\xi_t$ and $\tA$ is bounded, the Optional Stopping Theorem implies that $\EE\left[Z_{\tA}\right] = 0$. Therefore, 
\begin{align*}
\EE\left[\sum_{t=1}^{{\tA}} \mu_t^\top \left(a_t - b_{t} (x_t)\right) \right] 
= \EE\left[\sum_{t=1}^{{\tA}} \mathbb E\left[ \mu_t^\top \left(a_t - b_{t} (x_t)\right) | \sigma(\xi_{t-1}) \right] \right]\,.
\end{align*}
Using a similar martingale argument for $f_t(x_t) + r(a_t)$ and summing \eqref{eq:bound_one_period} from $t=1,\ldots,\tA$ we obtain that
\begin{align}\label{eq:f_and_r}
\mathbb E\left[ \sum_{t=1}^{\tA} f_t(x_t) + r(a_t) \right] 
&= 
\mathbb E\left[ \sum_{t=1}^{\tA} \bar D(\mu_t| \cP) \right ]  - \mathbb E\left[ \sum_{t=1}^{\tA} \mu_t^\top \left(  a_t - b_{t} (x_t) \right) \right] \nonumber\\   
&\ge 
\mathbb E\left[ \tA \bar D(\bmu_{\tA}| \cP) \right ] - \mathbb E\left[ \sum_{t=1}^{\tA} \mu_t^\top \left( a_t - b_{t} (x_t) \right) \right]\,.  
\end{align}
where the inequality follows from denoting $\bmu_{\tA} = \frac 1 {\tA} \sum_{t=1}^{\tA} \mu_t$ to be the average dual variable and using that the dual function is convex.

\paragraph{Step 2 (Complementary slackness).} Consider the sequence of functions
\begin{equation}\label{eq:wt}
    w_t(\mu) = \mu^\top(a_t - b_t(x_t))\,,
\end{equation}
which capture the complementary slackness at time $t$. The gradients are given $\nabla_\mu w_t(\mu) = a_t - b_t(x_t)$, {\color{black} which are bounded as follows $\|\nabla_\mu w_t(\mu) \|_{\infty} \le \|b_t(x_t)\|_{\infty} + \|a_t\|_{\infty} \le \ubb + \uba$}. Therefore, Algorithm~\ref{al:sg} applies online mirror descent to these sequence of functions $w_t(\mu)$, and we obtain from Proposition~\ref{prop:omd} that for every $\mu \in \cD$
\begin{align}\label{eq:regret_omd}
\sum_{t=1}^{\tA} w_t(\mu_t) - w_t(\mu) \le E(\tA, \mu) \le E(T,\mu)\,,
\end{align}
where $E(t,\mu) = \frac 1 {2\sigma} (\ubb + \uba)^2 \eta \cdot t + \frac{1}{\eta} V_h(\mu,\mu_0)$ is the regret of the online mirror descent algorithm after $t$ iterations, and the second inequality follows because $\tA \le T$ and the error term $E(t,\mu)$ is increasing in $t$.

We now discuss the choice of $\mu$. For $\hat \mu=\arg\max_{\mu\in\cD} \left\{ r^*(-\mu) - \mu^\top \left( \frac1 T \sum_{t=1}^{T} b_t(x_t) \right) \right\}$, we have that
\begin{align}\label{eq:choice_mu}
\frac 1 T \sum_{t=1}^T \left( r(a_t) + \hat \mu^\top a_t\right)  \le r^*(-\hat \mu) = r\left( \frac 1 T \sum_{t=1}^{T} b_t(x_t) \right) +  \hat \mu^\top \left( \frac1 T \sum_{t=1}^{T} b_t(x_t) \right)\,,
\end{align}
where the inequality follows because $r^*(-\hat\mu) = \max_{a\le \rho} \left\{ r(a) + \hat\mu^\top a \right\} \ge r(a_t) + \hat \mu^\top a_t$ because $a_t \le \rho$, and the equality is because $r^*(-\hat \mu) - \hat \mu^\top a = r(a)$ for $a = \frac 1 T \sum_{t=1}^{T} b_t(x_t)$ since $a\le\rho$ and $(r^*)^*(a) = r(a)$ since $r(a)$ is closed, concave, and proper by Assumption~\ref{ass:r2}.

We let $\mu = \hat \mu + \delta$, where $\delta \in \RR_+^m$ non-negative is to be determined later. Note that $\mu \in \cD$ because the positive orthant is inside the recession cone of $\cD$ (see Lemma \ref{lem:mu_negative}). Putting these together, we bound the complementary slackness as follows
\begin{align}\label{eq:bound-cs}
\begin{split}
\sum_{t=1}^{\tA} w_t(\mu_t)
&\le \sum_{t=1}^{\tA} w_t(\mu) + E(T,\mu)\\
&= \sum_{t=1}^{T} w_t(\hat \mu) - \sum_{t=\tA+
	1}^T w_t(\hat \mu) + \sum_{t=1}^{\tA} w_t(\delta) + E(T,\mu)\\
&\le T r\left( \frac 1 T \sum_{t=1}^{T} b_t(x_t) \right) - \sum_{t=1}^T r(a_t) - \sum_{t=\tA+1}^T w_t(\hat \mu)  + \sum_{t=1}^{\tA} w_t(\delta) + E(T,\mu)\,,
\end{split}
\end{align}
where the first inequality follows from \eqref{eq:regret_omd}, the  equality follows from linearity of $w_t(\mu)$, and the second inequality from \eqref{eq:choice_mu}.

\paragraph{Step 3 (Putting it all together).}
For any $\cP\in \Delta(\cS)$ and $\tA \in [0,T]$ we have that
\begin{align}\label{eq:bound-opt}
\EE_{\gamma_t \sim \cP} \left[ \OPT(\vgamma) \right] &=  \frac{\tA}{T} \EE_{\gamma_t \sim \cP} \left[ \OPT(\vgamma) \right] + \frac{T-\tA}{T}\EE_{\gamma_t \sim \cP} \left[ \OPT(\vgamma) \right] \nonumber \\ &\le  \tA \bar D(\bmu_{\tA}| \cP) + \pran{T-\tA}{(\ubf+\ubr)}\ ,
\end{align}
where the inequality uses \eqref{eq:obtain_dual} and the fact that $\OPT(\vgamma)\le \ubf+\ubr$. Let $\Regret{A|\cP} = \EE_{\gamma_t \sim \cP} \left[ \OPT(\vgamma) - R(A|\vgamma)\right]$ be the regret under distribution $\cP$. Therefore,
\begin{align*}
\Regret{A|\cP} &= \EEcP{ \OPT(\vgamma)-R(A|\vgamma)} \\
&\le \EEcP{ \OPT(\vgamma) - \sum_{t=1}^{\tA} f_t(x_t)-T\roneT}  \\
&\le \EEcP{ \OPT(\vgamma) - \tA D(\bmu_{\tA}| \cP)  +\sum_{t=1}^{\tA} \left(w_t(\mu_t) + r(a_t)\right) -T\roneT }\\  
&\le \EEcP{ \OPT(\vgamma) - \tA D(\bmu_{\tA}| \cP)  +\sum_{t=1}^{\tA} w_t(\delta) -
	\sum_{t=\tA+1}^{T} \left(w_t(\hat \mu) + r(a_t) \right) + E(T,\mu) } \\      
&\le \EE_{\cP} \Bigg[ \underbrace{ (T - \tA) \cdot \left( \bar f + \ubr + \|\hat \mu\|_1 (\ubb + \uba) \right) + \sum_{t=1}^{\tA} w_t(\delta) +E(T, \mu)}_{\clubsuit} \Bigg] \,,
\end{align*}
where the first inequality follows from using that $\tA \le T$ together with $f_t(\cdot) \ge 0$ to drop all requests after $\tA$; the second is from \eqref{eq:f_and_r}; the third follows from because \eqref{eq:bound-cs}; and the last because from \eqref{eq:bound-opt}, and using Cauchy-Schwartz together with the triangle inequality to obtain that {\color{black} $w_t(\hat \mu) = \hat \mu^\top (a_t - b_t(x_t)) \ge - \|\hat \mu\|_1 ( \|b_t(x_t)\|_{\infty} + \|a_t\|_{\infty}) \ge - \|\hat \mu\|_1 (\ubb + \uba)$} and $r(a_t) \ge 0$. 

We now discuss the choice of $\delta \in \RR_+^m$. Let $C = \ubf + \ubr + \|\hat \mu\|_1 (\ubb + \uba)$. If $\tA = T$, then set $\delta = 0$, and we obtain that $\clubsuit \le  E(T, \hat \mu)$. If $\tA < T$, then there exists a resource $j\in[m]$ such that $\sum_{t=1}^{\tA} \btj^{\top} x_t + {\ubb} \ge T \rho_j $. Set $\delta = (C/\rho_j) e_j$ with $e_j$ being the $j$-th unit vector. This yields
\begin{align*}
\sum_{t=1}^{\tA} w_t(\delta)
&=\sum_{t=1}^{\tA} \delta^\top (a_t - b_t(x_t))
= \frac C {\rho_j} \sum_{t=1}^{\tA} \left( (a_t)_j - (b_t)_j x_t  \right)\\
&\le \frac C {\rho_j} \left( \tA \rho_j - T \rho_j + \ubb \right)
= \frac C {\rho_j} \ubb - C (T - \tA)\,,
\end{align*}
where the inequality follows because $a_t \le \rho$ and the definition of the stopping time $\tA$. Therefore,
\[
 \clubsuit \le \frac{C \ubb}{\rho_j} + E(T,\mu) \le \frac{(\bar f + \ubr + \|\hat \mu\|_1 (\ubb + \uba))\ubb}{\lbrho} + \frac 1 {2 \sigma} (\ubb + \uba)^2 \eta\cdot T + \frac{1}{\eta} V_h(\mu, \mu_0) \,,
\]
where the second inequality follows from $\rho_j \ge \lbrho$, and our formulas for $C$ and $E(T,\mu)$.

We conclude by noting that $-\hat \mu$ is a supergradient of $r(a)$ at $a = \frac 1 T \sum_{t=1}^{T} b_t(x_t)$. This follows because, for every $a'\le\rho$, $r^*(-\hat \mu) \ge r(a') + \hat \mu^\top a'$, by definition of the conjugate function, and $r^*(-\hat \mu) - \hat \mu^\top a = r(a)$ yield $r(a') \le r(a) - \hat \mu^\top(a' - a)$. Therefore, $\|\hat \mu\|_1 \le L$ because $r(a)$ is $L$-Lipschitz continuous with respect to the norm $\|\cdot\|_{\infty}$ (see, for example, Lemma 2.6 in \citealt{shalev2012online}) and, additionally, the triangle inequality implies that $\|\mu\|_1 \le \|\hat \mu\|_1 + \| \delta\|_1 \le L + C/\lbrho$ by our choice of $\delta$, thus $V_h(\mu, \mu_0)\le \sup_{\mu \in \cD: \|\mu\|_1 \le L + C/\lbrho}  V_h(\mu, \mu_0)$.

The proof follows by combining the cases for $\tA = T$ and $\tA < T$.



\subsection{Proof of Theorem \ref{thm:adversial}}

We here discuss how the steps of the proof of Theorem~\ref{thm:master} need to be adapted to account for the adversarial requests. 

\textbf{Step 1 (Primal performance).}
Notice that $x_t=\arg\max_{x\in \cX_t} \{f_t(x)-\mu_t^\top b_t(x)\}$, thus we have $f_t(x_t) \ge f_t(x_t^*) - \mu_{t}^{\top} b_{t} (x_t^*) - b_{t} (x_t)$ and $0 = f_t(0) \le f_t(x_t) - \mu_t^\top b_t(x_t)$, whereby
\begin{align}\label{eq:alpha-time}
\begin{split}
    &\alpha(f_t(x_t)+r(a_t)) \\
     = & f_t(x_t)+r(a_t) + (\alpha-1)(f_t(x_t)+r(a_t)) \\
     \ge & f_t(x_t^*) + r(a^*) + (\mu_t^\top b_t(x_t) + r(a_t))-(\mu_t^\top b_t(x_t^*) + r(a^*))+ (\alpha-1)(\mu_t^\top b_t(x_t)+r(a_t)) \\
     = & f_t(x_t^*) + r(a^*) - \alpha \mu_t^\top ( a_t-b_t(x_t)) + \alpha (\mu_t^\top a_t + r(a_t)) - (\mu_t^\top b_t(x_t^*) + r(a^*)) \\
     =& f_t(x_t^*) + r(a^*) - \alpha \mu_t^\top ( a_t-b_t(x_t)) + \alpha r^*(-\mu_t) - (\mu_t^\top b_t(x_t^*) + r(a^*)) \ ,
\end{split}
\end{align}
where the last equality utilizes the definition of $r^*$ and $a_t$. Summing up \eqref{eq:alpha-time} over $t=1,\ldots,\tA$, we arrive at
\begin{align}\label{eq:sum-up}
    \begin{split}
        \alpha \sum_{t=1}^{\tau_A} (f_t(x_t)+r(a_t)) 
        \ge \sum_{t=1}^{\tau_A} f_t(x_t^*) + r(a^*) - \sum_{t=1}^{\tau_A} \alpha \mu_t^\top ( a_t-b_t(x_t)) +  \sum_{t=1}^{\tau_A} \alpha r^*(-\mu_t) - (\mu_t^\top b_t(x_t^*) + r(a^*)) \ .
    \end{split}
\end{align}
Noticing $r(a^*)\le r(0)+p^\top a^*$ and $a^*=\frac{1}{T} \sum_{t=1}^T b_t(x_t^*)$, it holds that
\begin{align}
    \begin{split}
        &\sum_{t=1}^{\tau_A} \alpha r^*(-\mu_t) - (\mu_t^\top b_t(x_t^*) + r(a^*)) \\
        \ge & \sum_{t=1}^{\tau_A} \alpha r^*(-\mu_t) - (\mu_t^\top b_t(x_t^*) + r(0) + p^\top a^*) \\
        = & \sum_{t=1}^{\tau_A} \alpha r^*(-\mu_t) - (\mu_t^\top b_t(x_t^*) +  r(0) +  p^\top b_t(x_t^*)) + \frac{T-\tau_A}{T} \sum_{t=1}^{\tau_A} p^\top b_t(x_t^*) -
        \frac{\tau_A}{T}
        \sum_{t=\tau_A+1}^{T} p^\top b_t(x_t^*)  \\
        \ge & \sum_{t=1}^{\tau_A} \alpha r^*(-\mu_t) - (\mu_t^\top b_t(x_t^*) +  r(0) +  p^\top b_t(x_t^*)) - 2\frac{T-\tau_A}{T} T\|p\|_1\ubb  \\
        \ge & - 2(T-\tA)\|p\|_1 \ubb \ ,
    \end{split}
\end{align}
where the second inequality utilizes $|p^\top b_t(x_t^*)|\le \|p\|_1\ubb$  and the last inequality utilizes \eqref{eq:adv_condition}. Substituting the above into \eqref{eq:sum-up}, we obtain
\begin{equation}\label{eq:target}
    \alpha \sum_{t=1}^{\tau_A} (f_t(x_t)+r(a_t)) 
        \ge \sum_{t=1}^{\tau_A} f_t(x_t^*) + r(a^*) - \alpha \sum_{t=1}^{\tau_A}  \mu_t^\top ( a_t-b_t(x_t)) - 2(T-\tA)\|p\|_1\ubb \ .
\end{equation}

\paragraph{Step 2 (Complementary slackness).} This step applies directly because the analysis is deterministic in nature.
In particular, it holds for $\hat \mu=\arg\max_{\mu\in\cD} \left\{ r^*(-\mu) - \mu^\top \left( \frac1 T \sum_{t=1}^{T} b_t(x_t) \right) \right\}$ and $\mu = \hat \mu + \delta$ with any $\delta\ge 0$ that
\begin{align}\label{eq:bound-cs-2}
\begin{split}
\sum_{t=1}^{\tA} w_t(\mu_t)
\le T r\left( \frac 1 T \sum_{t=1}^{T} b_t(x_t) \right) - \sum_{t=1}^T r(a_t) - \sum_{t=\tA+1}^T w_t(\hat \mu)  + \sum_{t=1}^{\tA} w_t(\delta) + E(T,\mu)\,,
\end{split}
\end{align}
where $w_t(\mu)$ is defined in \eqref{eq:wt} and $E(t,\mu)$ is the regret of the online algorithm as specified in \eqref{eq:regret_omd}.


\paragraph{Step 3 (Putting it all together).} Choosing $\mu = \hat \mu + \delta$ with $\hat \mu$ as given in \eqref{eq:choice_mu} and $\delta \ge 0$ gives
\begin{align*}
    & \OPT(\vec \gamma) - \alpha R(A | \vec \gamma) \\
    \le & \sum_{t=1}^{T} f_t(x_t^*) + T r(a^*) - \alpha \sum_{t=1}^{\tA} f_t(x_t) - \alpha T \roneT\\
    \le & \sum_{t=\tA + 1}^{T} f_t(x_t^*) + (T-\tA) r(a^*) +  \alpha \sum_{t=1}^{\tA} \pran{r(a_t)+ w_t(\mu_t)} - \alpha T \roneT + 2(T-\tA) \|p\|_1\ubb \\
    \le & \sum_{t=\tA + 1}^{T} f_t(x_t^*) + (T-\tA) r(a^*) - \alpha \sum_{t=\tA + 1}^{T} \left(w_t(\hmu)  + r(a_t) \right) + \alpha \sum_{t=1}^{\tA} w_t(\delta) + \alpha E(T, \mu)
     + 2(T-\tA)\|p\|_1 \ubb \\
    \le & (T-\tA) \cdot  C - \alpha \sum_{t=1}^{\tA} \delta^\top (b_t(x_t) - a_t) + \alpha E(T,\mu)\,,
\end{align*}
where the first inequality follows because $\tA \le T$ together with $f_t(\cdot) \ge 0$, the second inequality is from \eqref{eq:target}, the third inequality utilizes \eqref{eq:bound-cs-2}, and the last inequality utilizes $f_t(x^*_t) \le \ubf$, $r(a^*)\le \ubr$, $r(a_t)\ge 0$, $w_t(\hmu)\le \|\hmu\|_1\|b_t(x_t) - a_t\|_{\infty}\le L (\ubb +\uba)$ because $\|\hmu\|_1 \le L$, and setting $C=\ubf+\ubr+\alpha L(\ubb +\uba)+ 2\|p\|_1\ubb$ similarly to the stochastic case.

If $\tA = T$, then set $\delta = 0$, and the result follows. If $\tA < T$, then there exists a resource $j\in[m]$ such that $\sum_{t=1}^{\tA} \btj^{\top} x_t + {\ubb} \ge T \rho_j$. Set $\delta = (C/(\alpha \rho_j)) e_j$ with $e_j$ the unit vector and repeat the steps of the stochastic i.i.d. case to obtain:
\[
    \OPT(\vec \gamma) - \alpha R(A | \vec \gamma) \le \frac{C \ubb}{\rho_j} + \alpha E(T,\mu) 
\]
Recall that the regret bound of the mirror descent algorithm is given by $E(T,\mu) = \frac 1 {2\sigma} (\ubb + \uba)^2 \eta \cdot T + \frac{1}{\eta} V_h(\mu,\mu_0)$. The result follows by a similar analysis to the stochastic case because $\|\mu\|_1 \le \|\hat \mu\|_1 + \| \delta\|_1 \le L + C/(\alpha\lbrho)$ by our choice of $\delta$ together with $V_h(\mu, \mu_0)\le \sup_{\mu \in \cD: \|\mu\|_1 \le L + C/(\alpha \lbrho)}  V_h(\mu, \mu_0)$. \qed

\subsection{Competitive Ratios for the Examples}\label{sec:cr-examples}

The next corollary applies Theorem~\ref{thm:adversial} to the regularizer examples stated in Section~\ref{sec:ex}.

\begin{cor}\label{cor:competitive-ratio}The following hold:


\begin{itemize}
    \item \textbf{Example \ref{ex:no_regularizer} (No Regularizer):} If $r(a) = 0$, then Algorithm~\ref{al:sg} is $\max\{\alpha_0,1\}$-competitive with $\alpha_0 = \sup_{(f, b, \cX) \in \cS} \sup_{j \in [m], x\in\cX} b_j(x) / \rho_j $. 
    
    \item \textbf{Example \ref{ex:minimal_cons} (Load Balancing):} If $r(a) = \lambda\min_{j \in [m]} \big((\rho_j - a_j)/\rho_j\big)$, then Algorithm~\ref{al:sg} is $(\alpha_0+1)$-competitive.
    
    \item \textbf{Example \ref{ex:hinge_loss} (Below-Target Consumption):} If  $r(a)=\sum_{j=1}^m c_j \min(\rho_j-a_j, \rho_j-t_j)$, then Algorithm~\ref{al:sg} is $\alpha$-competitive with $\alpha  = \max\{\sup_{(f, b, \cX) \in \cS} \sup_{j \in [m], x\in\cX} b_j(x) / t_j, 1\}$.
    
    \item \textbf{Example \ref{ex:hinge_loss2} (Above-Target Consumption):} If $r(a)=\sum_{j=1}^m c_j \min(a_j, t_j)$, then Algorithm~\ref{al:sg} is $\alpha$-competitive with $\alpha = \max\{\sup_{(f, b, \cX) \in \cS} \sup_{j \in [m], x\in\cX} b_j(x) / t_j, 1\}$.
\end{itemize}

\end{cor}

\begin{proof}

\textbf{Example \ref{ex:no_regularizer} (No Regularizer):} If $r(a)=0$, then $\cD$ is the positive orthant, and, for $\mu \in\cD$, $r^*(-\mu)=\rho^\top \mu$. Using that $p=0 \in \partial r(0)$ and $r(0)=0$, we obtain
$$
\sup_{x\in\cX} \left\{ (\mu + p)^\top b(x) \right\} + r(0) = \sup_{x \in \cX} \mu^\top b(x)  \le \alpha_0 \rho^\top \mu = \alpha_0 r^*(-\mu) \ ,
$$
because $b_j(x) \le \alpha_0 \rho_j$ for all $x \in \cX$ and $\mu \ge 0$. We conclude by applying Theorem~\ref{thm:adversial}.

\textbf{Example \ref{ex:minimal_cons} (Load Balancing):} If $r(a) = \lambda-\lambda\max_{j \in [m]} (a_j/\rho_j)$, then $\cD=\big\{\mu \ge 0 \mid \sum_{j=1}^m \rho_j \mu_j \ge \lambda\big\}$, and, for $\mu \in\cD$,  $r^*(-\mu)= \rho^\top \mu$. Using that $p=0 \in \partial r(0)$ and $r(0)=\lambda$, we obtain
$$
\sup_{x\in\cX} \left\{ (\mu + p)^\top b(x) \right\} + r(0) = \sup_{x\in\cX} \mu^\top b(x) + \lambda \le 
\sup_{x\in\cX}\mu^\top(b(x) + \rho) \le \alpha \rho^\top \mu = \alpha r^*(-\mu) \ ,
$$
where the first inequality uses $\lambda\le \rho^\top \mu$ since $\mu\in\cD$, and the second that $\mu \ge 0$ and $b_j(x) + \rho_j \le \alpha \rho_j$ because $\alpha = \alpha_0 + 1$. This finishes the proof by applying Theorem \ref{thm:adversial}.

\textbf{Example \ref{ex:hinge_loss} (Below-Target Consumption):} If $r(a)=\sum_{j=1}^m c_j \min(\rho_j-a_j, \rho_j-t_j)$, $\cD=\mathbb R_+^m$ and, for $\mu \in \cD$, $r^*(-\mu) =  \mu^\top t + \sum_{j=1}^m (\rho_j - t_j) \max(\mu_j, c_j)$. Using that $p=0 \in \partial r(0)$ and $r(0)=\sum_j c_j (\rho_j-t_j)$, 
we obtain
\begin{align*}
    \sup_{x\in\cX} \left\{ (\mu + p)^\top b x \right\} + r(0)  &= \sup_{x\in\cX} \mu^\top b x + \sum_j c_j (\rho_j-t_j)
    \le  \alpha \mu^\top t + \sum_{j=1}^m (\rho_j - t_j) \max(\mu_j, c_j)\\
    \le& \alpha \pran{\mu^\top t + \sum_{j=1}^m (\rho_j - t_j) \max(\mu_j, c_j)} = \alpha r^*(-\mu)\ .
\end{align*}
where the first inequality uses that $b_j(x) \le t_j \alpha$ for all $x \in \cX$ from our choice of $\alpha$ together with $\mu \ge 0$ and $c_j \le \max(c_j, \mu_j)$ together with $t_j \le \rho_j$, and the last inequality because $\alpha \ge 1$. This finishes the proof by applying Theorem \ref{thm:adversial}.

\textbf{Example \ref{ex:hinge_loss2} (Above-Target Consumption):} If $r(a)=\sum_{j=1}^m c_j \min(a_j, t_j)$, then $\cD=\left\{\mu \in \RR^m \mid \mu \ge -c\right\}$ and, for $\mu \in \cD$, $r^*(-\mu) = c^\top t + \mu^\top t + \sum_{j=1}^m (\rho_j - t_j) \max(\mu_j, 0)$. Using that $p=c \in \partial r(0)$ and $r(0)=0$, 
we obtain
$$
\sup_{x\in\cX} \left\{ (\mu + p)^\top b(x) \right\} + r(0) = (\mu + c)^\top b(x) \le \alpha (\mu+c)^\top t \le \alpha r^*(-\mu) \ ,
$$
where the first inequality uses $\mu+c\ge 0$ together with $b_j(x) \le \alpha t_j$ for all $x \in \cX$ by our choice of $\alpha$, and the last inequality uses that $t_j \le \rho_j$. This finishes the proof by applying Theorem \ref{thm:adversial}.
\end{proof}

\subsection{Proof of Theorem \ref{thm:adv-fairness}}
\paragraph{Part 1 (online matching).} Let $\beta=\max_j \frac{1}{\rho_j}$ and $\mu^+=\max\{\mu, 0\}$ be the positive part of $\mu$ and $\mu^-=\min\{\mu, 0\}$ be the negative part of $\mu$, then $\mu=\mu^+ + \mu^-$. Notice that
\begin{equation}\label{eq:muxstar}
    -\mu_t^\top x_t^* \ge - (\mu_t^+)^\top x_t^* \ge -\beta (\mu_t^+)^\top a_t = -\beta \mu_t^\top a_t + \beta (\mu_t^-)^\top a_t \ge -\beta \mu_t^\top a_t - \beta \ ,
\end{equation}
where the second inequality uses $x_t\le \beta\rho$ from the definition of $\beta$ and $a_t=\rho$ and the last inequality uses $\mu_t\in \cD$ thus $(\mu_t^-)^\top a_t \ge -1$. Furthermore, it follows from $x_t=\arg\max_{x\in \cX} \{f_t(x)-\mu_t^\top x\}$ that 
\begin{equation}\label{eq:ftboundone}
    f_t(x_t)\ge f_t(x_t^*) + \mu_t^\top (x_t - x_t^*) \text{ and } f_t(x_t)\ge \mu_t^\top x_t \ ,
\end{equation}
by noticing $x_t^*\in\cX$ and $0\in \cX$. We claim that the following holds:
\begin{align*}
        (\beta+1)(f_t(x_t)+r(a_t)) 
    \ge  f_t(x_t^*) + r(a^*) + (\beta+1) \mu_t^\top (x_t - a_t)  \ ,
\end{align*}
which would immediately give rise to the $(\beta+1)$-competitive with the same proof as Theorem \ref{thm:adversial}. We prove this claim by considering two cases.

\textbf{Case 1.} We first consider the underbooked case, i.e., $\sum_j\rho_j\le 1$. Then $a_t = \rho\in\cX$, thus it follows from the $x_t=\arg\max_{x\in \cX} \{f_t(x)-\mu_t^\top x\}$ that 
\begin{equation}\label{eq:ftboundtwo}
    f_t(x_t)\ge f_t(a_t)+ \mu_t^\top(x_t-a_t)\ge \mu_t^\top(x_t-a_t)\ .
\end{equation}
{because $f_t(\cdot) \ge 0$.}
Therefore, we have
\begin{align}\label{eq:wewant}
\begin{split}
        & (\beta+1)(f_t(x_t)+r(a_t)) \\
    \ge & f_t(x_t^*) + \mu_t^\top(x_t-x_t^*) + (\beta-1) \mu_t^\top x_t + \mu_t^\top (x_t-a_t) + (\beta+1) r(a_t) \\
    \ge & f_t(x_t^*) + (\beta+1) \mu_t^\top x_t - \mu_t^\top a_t - \beta \mu_t a_t -\beta + (\beta+1) r(a_t) \\
    \ge & f_t(x_t^*) + (\beta+1) \mu_t^\top (x_t - a_t)  + r(a^*) \ ,
\end{split}
\end{align}
where the first inequality uses \eqref{eq:ftboundone} and \eqref{eq:ftboundtwo} and by writing $(\beta+1)f_t(x_t) = f_t(x_t) + f_t(x_t) + (\beta-1) f_t(x_t)$, the second inequality uses \eqref{eq:muxstar} and the third inequality is from $r(a_t)=r(\rho)=1\ge r(a^*)$.

\textbf{Case 2.} We now consider the overbooked case, i.e., $ \sum_j\rho_j\ge 1$. Let $q=\sum_j\rho_j$, then $\frac{1}{q}\rho\in\cX$, thus $f_t(x_t)\ge f_t(\frac{1}{q} a_t)+ \mu_t^\top(x_t-\frac{1}{q} a_t)\ge \mu_t^\top(x_t- \frac{1}{q} a_t)$. Furthermore, we have
$$r(a^*)=\min_j\frac{ a^*_j}{\rho_j} \le \frac{\sum_j a^*_j}{ \sum_j \rho_j}\le \frac{1}{q} \ ,$$
where the last inequality utilizes $a^*\in\cX$, thus $\sum_j a^*_j\le 1$. Therefore,
\begin{equation}\label{eq:using_q}
    f(x_t)+ r(a_t) \ge \mu_t^\top(x_t-\frac{1}{q}a_t) + r(a^*) + \left(1-\frac{1}{q}\right) \ge \mu_t^\top(x_t-a_t) + r(a^*) \ ,
\end{equation}
where the first inequality uses $r(a_t)=1\ge r(a^*)+\left(1-\frac{1}{q}\right)$ and the second inequality is from $-\mu_t^\top a_t \le 1$ {because $\mu_t \in \cD$}.
Similar to \eqref{eq:wewant}, we have
\begin{align}\label{eq:wewant2}
\begin{split}
        & (\beta+1)(f_t(x_t)+r(a_t)) \\
    \ge & f_t(x_t^*) + \mu_t^\top(x_t-x_t^*) + (\beta-1) \mu_t^\top x_t + \mu_t^\top (x_t-a_t) + \beta r(a_t) + r(a^*) \\
    \ge & f_t(x_t^*) + (\beta+1) \mu_t^\top x_t - \mu_t^\top a_t - \beta \mu_t a_t -\beta + \beta r(a_t) + r(a^*) \\
    \ge & f_t(x_t^*) + (\beta+1) \mu_t^\top (x_t - a_t)  + r(a^*) \ ,
\end{split}
\end{align}
where the first inequality uses \eqref{eq:using_q} and \eqref{eq:ftboundone} {and by writing $(\beta+1)(f_t(x_t)+r(a_t)) = f_t(x_t) + (\beta-1) f_t(x_t) + \big(f_t(x_t)+r(a_t)\big) + \beta r(a_t)$}, the second inequality uses \eqref{eq:muxstar} and the third inequality is from $r(a_t)=r(\rho)=1\ge r(a^*)$.




\paragraph{Part 2 (generic online allocation).} Let $\mu^+$ be the positive part of $\mu$ and $\mu^-$ be the negative part of $\mu$, then $\mu=\mu^+ + \mu^-$. Notice that
\begin{equation}\label{eq:muxstar2}
    -\mu_t^\top b_t(x_t^*) \ge - (\mu_t^+)^\top b_t(x_t^*) \ge -\beta_2 (\mu_t^+)^\top a_t = -\beta_2 \mu_t^\top a_t + \beta_2 (\mu_t^-)^\top a_t \ge -\beta_2 \mu_t^\top a_t - \beta_2 \ge -\beta \mu_t^\top a_t - \beta \ ,
\end{equation}
where the second inequality is from $b_t x_t\le \beta_2 \rho$ by noticing the definition of $\beta_2$ and $a_t=\rho$, the third inequality uses $\mu_t\in \cD$, the last inequality uses $\mu_t\in\cD$ thus $(\mu_t^-)^\top a_t\ge -1$. Furthermore, it follows from $x_t=\arg\max_{x\in \cX_t} \{f_t(x)-\mu_t^\top x\}$ that 
\begin{equation}\label{eq:ftboundthree}
    f_t(x_t)\ge f_t(x_t^*) + \mu_t^\top b_t(x_t - x_t^*) \text{ and } f_t(x_t)\ge \mu_t^\top b_t x_t \ ,
\end{equation}
by noticing $x_t^*\in\cX_t$ and $0\in \cX_t$. As before we claim the following holds:
\begin{align*}
        (\beta+1)(f_t(x_t)+r(a_t)) 
    \ge  f_t(x_t^*) + r(a^*) + (\beta+1) \mu_t^\top (b_t x_t - a_t)  \ ,
\end{align*}
which would give rise to the $(\beta+1)$-competitive with the same proof as Theorem \ref{thm:adversial}. We prove the claim by considering two cases. 

\textbf{Case 1.} Suppose $a_t = \rho\in b_t \cX_t$. Then it holds from the $x_t=\arg\max_{x\in \cX_t} \{f_t(x)-\mu_t^\top b_t x\}$ that 
\begin{equation}\label{eq:ftboundfour}
    f_t(x_t)\ge f_t(a_t)+ \mu_t^\top(b_t x_t-a_t)\ge \mu_t^\top(b_t x_t-a_t)\ .
\end{equation}

Therefore, we have
\begin{align}\label{eq:wewant3}
\begin{split}
        & (\beta+1)(f_t(x_t)+r(a_t)) \\
    = & (\beta_2+1)(f_t(x_t)+r(a_t)) + (\beta-\beta_2)(f_t(x_t)+r(a_t)) \\
    \ge & f_t(x_t^*) + \mu_t^\top b_t (x_t-x_t^*) + (\beta_2 -1) \mu_t^\top b_t x_t + \mu_t^\top b_t (x_t-a_t) + (\beta_2+1) r(a_t) + (\beta-\beta_2)(f_t(x_t)+r(a_t)) \\
    \ge & f_t(x_t^*) + (\beta_2+1) \mu_t^\top b_t x_t - \mu_t^\top a_t - \beta_2 \mu_t a_t -\beta_2 + (\beta_2+1) r(a_t) + (\beta-\beta_2)(f_t(x_t)+r(a_t)) \\
    \ge & f_t(x_t^*) + (\beta_2+1) \mu_t^\top (b_t x_t - a_t)  + r(a^*) + (\beta-\beta_2)(f_t(x_t)+r(a_t)) \\
    \ge & f_t(x_t^*) + (\beta+1) \mu_t^\top (b_t x_t - a_t)  + r(a^*) \ ,
\end{split}
\end{align}
where the first inequality uses \eqref{eq:ftboundthree} and \eqref{eq:ftboundfour}, the second inequality uses \eqref{eq:muxstar}, the third inequality is from $r(a_t)=r(\rho)=1\ge r(a^*)$ and the last inequality uses $f_t(x_t)+r(a_t)\ge \mu_t^\top b_t x_t + 1 \ge \mu_t^\top (b_t x_t - a_t)$.

\textbf{Case 2.} Suppose $\rho\not\in b_t \cX_t$, then we know $\frac{1}{\beta}\rho\in b_t \cX_t$, thus 
\begin{equation}\label{eq:useq2}
    f_t(x_t)\ge f_t(\frac{1}{\beta} a_t)+ \mu_t^\top(b_t x_t-\frac{1}{\beta} a_t)\ge \mu_t^\top(b_t x_t- \frac{1}{\beta} a_t)\ .
\end{equation}

Therefore, we have
\begin{align}\label{eq:wewant4}
\begin{split}
        & (\beta+1)(f_t(x_t)+r(a_t)) \\
    \ge & f_t(x_t^*) + \mu_t^\top b_t(x_t-x_t^*) + \beta \mu_t^\top (b_t x_t- \frac{1}{\beta} a_t) + (\beta + 1) r(a_t)  \\
    \ge & f_t(x_t^*) + (\beta+1) \mu_t^\top b_t x_t - \mu_t^\top a_t - \beta \mu_t a_t -\beta + \beta r(a_t) + r(a^*) \\
    \ge & f_t(x_t^*) + (\beta+1) \mu_t^\top (x_t - a_t)  + r(a^*) \ ,
\end{split}
\end{align}
where the first inequality uses \eqref{eq:ftboundthree} and \eqref{eq:useq2}, the second inequality uses \eqref{eq:muxstar2}. \qed

\section{Online Mirror Descent}

We reproduce a standard result on online mirror descent for completeness.

\begin{prop}\label{prop:omd}
Consider the sequence of convex functions $w_t(\mu)$. Let $g_t \in \partial_\mu w_t(\mu_t)$ be a subgradient and
   \begin{equation*}
       \mu_{t+1} = \arg\min_{\mu\in\cD} \langle g_t, \mu \rangle + \frac{1}{\eta} V_h(\mu, \mu_t)\,.
   \end{equation*}
Suppose subgradients are bounded by $\| g_t \|_{\infty} \le G$ and the reference function $h$ is $\sigma$-strongly convex with respect to $\|\cdot\|_{1}$-norm. Then, for every $\mu \in \cD$ we have
\begin{align*}
    \sum_{t=1}^{T} w_t(\mu_t) - w_t(\mu) \le \frac{G^2 \eta}{2\sigma}  T + \frac 1 \eta V_h(\mu,\mu_0)\,.
\end{align*}
\end{prop}

\end{document}